\documentclass{mfat}
\pagespan{220}{244}

\usepackage{mmap}
\usepackage[utf8]{inputenc}

\newcommand{\Q}{\mathbb{Q}}
\newcommand{\R}{\mathbb{R}}
\newcommand{\C}{\mathbb{C}}
\newcommand{\N}{\mathbb{N}}
\newcommand{\Z}{\mathbb{Z}}
\newcommand{\K}{\mathbb{K}}
\newcommand{\B}{\mathbb{B}}
\newcommand{\E}{\mathbb{E}}
\newcommand{\e}{\varepsilon}
\newcommand{\esssup}{\text{ess sup}}

\def\one {1\kern -1mm {\text {\rm I}}}

\newcommand{\bfrac}[2]{\genfrac{}{}{0pt}{}{#1}{#2}}

\newcommand{\limn}{\lim_{n \rightarrow \infty}}
\newcommand{\abs}{\overline}

\newtheorem{Theorem}{Theorem}[section]
\newtheorem{Proposition}[Theorem]{Proposition}
\newtheorem{Corollary}[Theorem]{Corollary}
\newtheorem{Lemma}[Theorem]{Lemma}
\newtheorem{Remark}[Theorem]{Remark}

\numberwithin{equation}{section}

\begin{document}

\title[Non-autonomous interacting particle systems in continuum]
    {Non-autonomous interacting particle systems in continuum}

\author{Martin Friesen}
\address{Department of Mathematics, University of Bielefeld, Bielefeld, Germany}
\email{mfriesen@math.uni-bielefeld.de}


\date{24/03/2016; \ \  Revised 15/06/2016}

\begin{abstract}
 A conservative Feller evolution on continuous bounded functions is constructed from a weakly continuous, time-inhomogeneous transition function describing a pure jump process
 on a locally compact Polish space. The transition function is assumed to satisfy a Foster-Lyapunov type condition.
 The results are applied to interacting particle systems in continuum, in particular to general birth-and-death processes (including jumps).
 Particular examples such as the BDLP and Dieckmann-Law model are considered in the end.
\end{abstract}

\maketitle

\vspace*{-3mm}
\section{Introduction}
Classical birth-and-death dynamics are described by a system of ordinary differential equations, also known as Kolmogorov's differential equations,
and are usually studied by semigroup methods on (weighted) spaces of summable real-valued sequences, cf. \cite{KATO54, FELLER68, FELLER71, HP74}.
More recent attempts study such equations on the spaces $\ell^p$ for $p \in [1,\infty)$, see \cite{BA06, BLM06, TV06}.
In contrast to many real world models, see e.g. the Bolker, Dieckmann, Law, Pacala model \cite{BP97, BP99, DL00, DL05} (short BDLP),
such equations do not include the positions of the described particles. Other models coming from ecology and the modeling of mutations can be found
in \cite{N01, BCFKKO14, KM66, SEW05,FFHKKK15} and references therein.

The simplest possibility to include spatial structure is to assign to
each particle a fixed site of a graph (e.g. from the lattice
$\Z^d$). This are the so-called lattice models. For such models a
rigorous study by semigroup methods is adequate and a detailed
presentation can be found in the classical book \cite{LIGGETT} and
references therein.  Several models, such as the BDLP model, require
that the positions of the particles be not a priori fixed. This means
that $\Z^d$ should be replaced by a continuous location space,
e.g. $\R^d$.

For the modeling of interacting particle systems in continuum the theory of pure point processes is commonly used.
Such processes share several properties with the processes associated to lattice models,
but also include numerous unexpected features and require essentially different techniques for their mathematical treatment.
Taking into account that they describe real-world particles it leads to the natural assumption
that all particles are indistinguishable and any two particles cannot occupy the same position in the location space, say for simplicity $\R^d$.
A microscopic state $\eta$ is then, by definition, a finite linear combination of point-masses $\delta_x$, where $x \in \R^d$ is the position of a particle in the system.
Such Markov dynamics can be analyzed by a measure-valued generalization of the Kolmogorov's differential equations.
This equations have been first analyzed in \cite{FELLER40} and have been afterwards further investigated in the next 60 years, cf. \cite{FMS14} and many others.
A summary with applications to interacting particle systems is provided in the book \cite{MU-FA04}.
In this work we identify $\eta$ with a subset of $\R^d$, i.e. we consider the microscopic state as a
finite collection of positions $x \in \R^d$. The state space (= configuration space) is therefore the space of all finite configurations, which is given by
\[
 \Gamma_0 = \{ \eta \subset \R^d \ | \ |\eta| < \infty\}.
\]
Here and in the following we write $|A|$ for the number of elements in $A \subset \R^d$.

\medskip

Interacting particle systems in continuum with state space $\Gamma_0$ are heuristically described by a Markov (pre-)generator on a proper set of functions $F$.
The general form of such operator is given by the heuristic expression

\begin{align}\label{INTRO:00}
 (LF)(\eta) = \sum _{\xi \subset \eta}\int _{\Gamma_0}(F(\eta \backslash \xi \cup \zeta)
 - F(\eta))K(\xi, \eta, d\zeta), \quad \eta \in \Gamma_0,
\end{align}
where $K(\xi,\eta,d\zeta) \geq 0$ is a transition kernel which will be specified in the third section.
The associated (backward) Kolmogorov equation
\[
 \frac{\partial F_t}{\partial t} = LF_t, \quad F_t|_{t=0} = F_0
\]
can be solved on the space of continuous bounded functions, see \cite{EW01, EW03, KOL06}.
The solutions determine therefore a Markov process on $\Gamma_0$. This process describes the evolution of the microscopic state $\eta$,
where each group $\xi \subset \eta$ of particles may disappear and simultaneously a new group of particles $\zeta \in \Gamma_0$ may appear somewhere in $\R^d$.
The distribution of the new particles and the intensity of this event are both described by the transition kernel $K(\xi,\eta, d\zeta)$.
To this end we define for any $\eta \in \Gamma_0$ and Borel measurable set $A \subset \Gamma_0$ a new transition kernel by
\[
 Q(\eta,A) := \sum _{\xi \subset \eta} \int _{\Gamma_0}\one_A(\eta \backslash \xi \cup \zeta)
 K(\xi,\eta, d\zeta),
\]
where $\one_{A}(\eta) := \begin{cases} 1, & \eta \in A \\ 0, & \eta \not \in A \end{cases}$.
Then, under some additional conditions, it is possible to rewrite above operator to
\begin{align}\label{INTRO:01}
 (LF)(\eta) = \int _{\Gamma_0}(F(\xi) - F(\eta))Q(\eta,d\xi), \quad \eta \in \Gamma_0.
\end{align}
In such a case we have to assume that
\begin{align*}
 q(\eta) := Q(\eta, \Gamma_0) = \sum _{\xi \subset \eta} K(\xi, \eta, \Gamma_0)
\end{align*}
is finite for all $\eta \in \Gamma_0$.
Hence the process described by the operator $L$ is a pure jump Markov process and techniques coming from the theory of Markov chains are applicable.
Such approach has been investigated in the last 20 years, a comprehensive summary of the obtained results can be found in \cite{MU-FA04}.
In the case of birth-and-death dynamics, i.e. a particular choice of $K(\xi,\eta,d\zeta)$, an alternative approach
is provided by solving certain stochastic equations, cf. \cite{BEZ152, BEZ151}.

\medskip

This work is organized as follows. The second section is devoted to the study of an abstract pure jump process described by the time-dependent operator
\[
 (L(t)F)(x) = \int _{E}(F(y) - F(x))Q(t,x,dy), \quad x \in E, \quad t \geq 0.
\]
Here $E$ is an abstract locally compact Polish space and it is assumed that the operator $L(t)$ satisfies a Foster-Lyapunov type condition, see \cite{MT93}.
Moreover, we suppose that the transition function $Q(t,x,dy)$ is weakly continuous and satisfies some additional technical conditions.
Based on the classical works \cite{FELLER40, GS75, FMS14}, we construct an associated conservative Feller evolution system $U(s,t)$
on the space of continuous bounded functions.
Hence by \cite{VANCASTEREN11} this Feller evolution is associated to a Hunt process with state space $E$.
For a countable state space $E$ such result was obtained by martingale techniques in \cite{ZZ87}.
We show that $U(s,t)$ provides existence and uniqueness of solutions to the Kolmogorov equations and establish
the relation to the jump process by the associated Martingale problem.
Additional works dealing with properties of (time-homogeneous) Markov processes can be found, e.g., in \cite{CN82, CN83, CHEN86, ZS99}.

The third section is devoted to particular examples of interacting particle systems on $\Gamma_0$.
We apply above results for time-dependent (pre-)generators given by
\begin{align}\label{INTRO:02}
 (L(t)F)(\eta) = \sum _{\xi \subset \eta}\int _{\Gamma_0}(F(\eta \backslash \xi \cup \zeta) - F(\eta))K_t(\xi, \eta,d\zeta).
\end{align}
To this end we use a similar representation to \eqref{INTRO:01} for the time-dependent transition function
\[
 Q(t,\eta, A) = \sum _{\xi \subset \eta}\int _{\Gamma_0}\one_A(\eta \backslash \xi \cup \zeta)
 K_t(\xi,\eta,d\zeta), \quad t \geq 0, \quad \eta \in \Gamma_0.
\]
As a consequence we are able to construct, under some reasonable conditions, an associated conservative Feller evolution system $U(s,t)$ on continuous bounded
functions $F: \Gamma_0 \longrightarrow \R$.
Hence for any initial probability measure $\mu$ on $\Gamma_0$ the action of the adjoint evolution system $U^*(t,s)\mu$
yields a weak solution to the Fokker-Planck equation
\[
 \frac{\partial }{\partial t}\int _{\Gamma_0}F(\eta)\mu_t(d\eta) = \int _{\Gamma_0}L(t)F(\eta)\mu_t(d\eta),
 \quad t \geq s,
\]
where $F: \Gamma_0 \longrightarrow \R$ is continuous and bounded.
Sufficient conditions for which $L^1(\Gamma_0, d\lambda)$ is invariant for $U^*(t,s)$ are given.
By construction, the restriction $U^*(t,s)|_{L^1(\Gamma_0, d\lambda)}$ becomes strongly continuous.

The rest of the third section is devoted to the study of particular examples.
First we consider the BDLP model with time-dependent and non-translation invariant kernels.
The considerations are afterwards extended to the Dieckmann-Law model and a generalization of this model.

\section{General jump processes}
Let $E$ be a locally compact Polish space and denote by $\mathcal{B}(E)$ the Borel-$\sigma$-algebra on $E$.
Denote by $BM(E)$ the Banach space of all bounded measurable functions and by $C_b(E)$ the subspace of all continuous bounded functions.
A pure jump process is determined by its (infinitesimal) transition function, i.e. a function $Q: \R_+ \times E \times \mathcal{B}(E) \longrightarrow \R_+$
with the following properties:
\begin{enumerate}
 \item[]{1.} For all $t \geq 0$, $x \in E$, $A \longmapsto Q(t,x,A)$ is a finite Borel measure with  \break  $Q(t,x,\{x\}) = 0$.
 \item[]{2.} For all $A \in \mathcal{B}(E)$, $(t,x) \longmapsto Q(t,x,A)$ is measurable.
 \item[]{3.} For all $T > 0$ and all compacts $B \subset E$
  \begin{align}\label{GJP:11}
   \sup _{(t,x) \in [0,T]\times B} Q(t,x,E) < \infty.
  \end{align}
\end{enumerate}
Let $BM_{loc}(E)$ be the space of locally bounded measurable functions and $C(E)$ be the space of continuous functions.
Define for any $F: E \longrightarrow \R$
\[
 (Q(t)F)(x) := \int _{E}F(y)Q(t,x,dy), \quad t \geq 0, \quad x \in E,
\]
whenever it makes sense, i.e. $\int _{E}|F(y)|Q(t,x,dy) < \infty$ for all $t \geq 0$, $x \in E$. Then $Q(t): BM(E) \longrightarrow BM_{loc}(E)$
is a well-defined positive linear operator and $q(t,x) := (Q(t)1)(x) = Q(t,x,E)$ is locally bounded.
If $Q(\cdot) C_b(E) \subset C(\R_+ \times E)$, i.e. for any $F \in C_b(E)$ the function $Q(\cdot)F$ is jointly continuous in $(t,x)$,
then we say that $Q$ is jointly continuous. This simply means that, by definition, $Q(t,x,dy)$ is weakly continuous in $(t,x)$.
In such a case \eqref{GJP:11} is automatically satisfied.

We briefly recall the results obtained in \cite{FELLER40, FMS14}. Let $Q$ be a transition function.
For $0 \leq s \leq t, x \in E$ and $A \in \mathcal{B}(E)$ let $P^{(0)}(s,x;t,A) := \delta(x,A)e^{-\int _{s}^{t}q(r,x)dr}$, and for $n \geq 1$
\begin{align}\label{GJP:03}
 P^{(n+1)}(s,x;t,A) := \int _{s}^{t}e^{-\int _{s}^{r}q(\tau,x)d\tau}\Big (\int _E P^{(n)}
 (r,y;t,A)Q(r,x,dy)\Big)dr.
\end{align}
Here $\delta(x,A) := \one_A(x) = \delta_x(A)$. Then $P(s,x;t,A) = \sum _{n=0}^{\infty}P^{(n)}(s,x;t,A)$ is a sub-Markov transition function.
Moreover, for fixed $A \in \mathcal{B}(E)$ and $x \in E$ it is absolutely continuous in $s$ and $t$, respectively such that $P(s,x;t,A) \to \delta(x,A)$ holds
uniformly in $A \in \mathcal{B}(E)$ whenever $s \to t^-$ or $t \to s^+$. For any $A \in \mathcal{B}(E)$ it is a.e. differentiable in $s \in [0,t]$ and satisfies
\begin{align}\label{GJP:00}
 \frac{\partial P(s,x;t,A)}{\partial s} = q(s,x)P(s,x;t,A) - \int _{E}P(s,y;t,A)Q(s,x,dy).
\end{align}
Likewise, for any compact $A \subset E$ it is differentiable for a.a. $t \in [s, \infty)$ and satisfies
\begin{align}\label{GJP:01}
 \frac{\partial P(s,x;t,A)}{\partial t} = - \int _{A}q(t,y)P(s,x;t,dy) + \int _{E}Q(t,y,A)P(s,x,t,dy).
\end{align}
It follows from \cite{FMS14} that $P$ is the minimal solution to \eqref{GJP:00} and \eqref{GJP:01}. Moreover, if $P(s,x;t,E) = 1$,
then this solution is also unique.

The main point of our interest is to study the (sub-)Markovian evolution system
\begin{align}\label{GJP:23}
 U(s,t)F(x) := \int _{E}F(y) P(s,x;t,dy), \quad 0 \leq s \leq t
\end{align}
on the space of bounded measurable functions and extensions of it. Such an evolution system is a family of positive bounded linear
operators such that $U(s,s)F = F$ and $U(s,r)U(r,t)F = U(s,t)F$ for $0 \leq s \leq r \leq t$. For $F \in BM(E)$ let
\begin{align}\label{GJP:24}
 L(t)F(x) = \int _{E}(F(y) - F(x))Q(t,x,dy), \quad t \geq 0
\end{align}
be the (formal) generator of $U(s,t)$. Since in general $U(s,t)F$ is not continuous w.r.t. the norm on $BM(E)$ or $C_b(E)$, we cannot expect that some extension
of $L(t)$ is a generator. However, for our needs it is sufficient to consider only the weaker concept of pointwise generator
The precise statement is given in the proposition below.

Denote by $\mathcal{C}$ the collection of compact sets on $E$
and by $\mathcal{C}_1$ the collection of compacts in $\R_+ \times E$. For a given non-negative function $V \in C(E)$ let
$\Vert F \Vert_V := \sup _{x \in E}\frac{|F(x)|}{1 + V(x)}$ and denote by $BM_V(E)$ the space of all measurable functions for which $\Vert F \Vert_V$ is finite.
Denote by $C_V(E) := BM_V(E) \cap C(E)$ its closed subspace of continuous functions. Below we state the main result for this section.
\begin{Proposition}\label{GJPTH:01}
 Assume that there exists a continuous function $V: E \longrightarrow \R_+$ such that $(t,x) \longmapsto Q(t)F(x)$ is continuous for any $F \in C_V(E)$.
 Moreover, suppose that there exists a continuous function $c: \R_+ \longrightarrow \R_+$ such that the properties below are satisfied.
 \begin{enumerate}
  \item[]{1.}For all $T > 0$ there exists $a(T) > 0$ such that $q(t,x) \leq a(T)V(x)$ holds for all $t \in [0,T]$ and $x \in E$.

  \item[]{2.} The Foster-Lyapunov estimate
  \begin{align}\label{GJP:12}
   \int _{E}V(y)Q(t,x,dy) \leq c(t)V(x) + q(t,x)V(x), \quad t \geq 0, \quad  x \in E
  \end{align}
  is satisfied.

  \item[]{3.} For all $\e > 0$, $B \in \mathcal{C}$ and $T > 0$ there exists $A \in \mathcal{C}$ such that
   \begin{align}\label{GJP:07}
    \int _{0}^{T}Q(r,x,A^c)\,dr < \e, \quad x \in B
  \end{align}
  is fulfilled.
 \end{enumerate}

 Then $U(s,t)$ is a conservative Feller evolution system, i.e. $U(s,t)1 = 1$ and $$(s,t,x) \longmapsto U(s,t)F(x)$$ is continuous for any $F \in C_b(E)$.
 Moreover, $U(s,t)$ can be extended to $BM_V(E)$ so that
 \begin{align}\label{GJP:13}
  \Vert U(s,t)F\Vert_V \leq e^{\int _{s}^{t}c(r)dr} \Vert F \Vert_V, \quad 0 \leq s \leq t.
 \end{align}
 The relation to the Kolmogorov equations is given by the statements below:
 \begin{enumerate}
  \item[]{(a)} For any $F \in BM(E)$, $t > 0$ and $x \in E$, $[0,t] \ni s \longmapsto U(s,t)F(x)$ is continuously differentiable and a solution to
   \begin{align}\label{GJP:16}
    \frac{\partial}{\partial s}U(s,t)F(x) = - L(s)U(s,t)F(x).
   \end{align}
   If in addition $F \in C_V(E)$, then $s \longmapsto U(s,t)F(x)$ is absolutely continuous and satisfies
   \eqref{GJP:16} a.e.

  \item[]{(b)} Let $F \in BM(E)$.
    Then for any $x \in E$, $s \geq 0$, $[s,\infty) \ni t \longmapsto U(s,t)F(x)$ is absolutely continuous and satisfies for a.a. $t \geq s$
   \begin{align}\label{GJP:20}
    \frac{\partial}{\partial t}U(s,t)F(x) = U(s,t)L(t)F(x).
   \end{align}

  \item[]{(c)} Let $V(s,t)$ be a Feller evolution system on $C_b(E)$. If for any $F \in C_b(E)$, $V(s,t)F$ is a solution to \eqref{GJP:16} or \eqref{GJP:20},
  then $V(s,t) = U(s,t)$ holds.
 \end{enumerate}
\end{Proposition}

\medskip

The time-homogeneous case was, e.g., treated in \cite{MU-FA04, KOL06}. Condition \eqref{GJP:12} can be reformulated to
\[
 \int _{E}(V(y) - V(x))Q(t,x,dy) \leq c(t)V(x), \quad t \geq 0, \quad x \in E.
\]
A transition function $Q$ with property \eqref{GJP:07} is said to have the localization property.
Property (c) means that $U(s,t)$ is the unique Feller evolution system associated with the operator $L(t)$.
The rest of this section is devoted to the proof of above statement.

Suppose from now on the conditions given in Proposition \ref{GJPTH:01} to be satisfied and let $\alpha \in (0,1)$.
Applying the iteration \eqref{GJP:03} to $(q(t,x), \alpha Q(t,x,dy))$ yields the sub-probability function given by
\begin{align}\label{GJP:26}
 P_{\alpha}(s,x;t,dy) = \sum _{n=0}^{\infty}\alpha^n P^{(n)}(s,x;t,dy).
\end{align}
Let $U_n(s,t)F(x) := \int _{E}F(y)P^{(n)}(s,x;t,dy)$, then $U_{\alpha}(s,t)F(x) := \sum _{n=0}^{\infty}\alpha^n U^{(n)}(s,t)F(x)$
defines an evolution system. We will call $U_{\alpha}(s,t)$ the regularized evolution system associated to $Q$. Clearly, above series converges uniformly in $(s,t,x)$.
The next lemma establishes the Feller property for $U_{\alpha}(s,t)$, whereas the limit $\alpha \to 1$ will be considered at the end of this section.
\begin{Lemma}\label{GJPTH:00}
 $(U_{\alpha}(s,t))_{0 \leq s \leq t}$ is a Feller evolution on $C_b(E)$.
\end{Lemma}
\begin{proof}
 It suffices to show that for any $n\geq 0$ and each $F \in C_b(E)$ the
 function \break
 $U^{(n)}(s,t)F(x)$ is continuous in all variables.
 Since $U^{(0)}(s,t)F(x) = F(x)e^{- \int _{s}^{t}q(r,x)dr}$, by Lemma \ref{LEMMAFELLER} this clearly holds for $n=0$.
 Assume the assertion holds for some $n \geq 0$. By \eqref{GJP:03} we get
 \begin{align}\label{GJP:04}
  U^{(n+1)}(s,t)F(x) = \int _{s}^{t}\int _E e^{- \int _{s}^{r}q(\tau,x)d\tau}(U^{(n)}(r,t)
  F)(y)Q(r,x,dy)\,dr.
 \end{align}
 By induction hypothesis $e^{-\int _{s}^{r}q(\tau,x)d\tau}(U^{(n)}(r,t)F)(y)$ is continuous in all variables.
 Moreover, due to $|U^{(n)}(r,t)F(y)| \leq \Vert F \Vert_{\infty}$ this function is bounded and hence by Lemma \ref{LEMMAFELLER2} we see that also
 \[
  (s,t,x) \longmapsto \int _E e^{-\int _{s}^{r}q(\tau,x)d\tau}(U^{(n)}(r,t)F)(y)Q(r,x,dy)
 \]
 is continuous. Thus Lemma \ref{LEMMAFELLER} yields the continuity of $U^{(n+1)}(s,t)F(x)$ in the variables $(s,t,x)$.
\end{proof}

The next result studies stability of the Feller evolution $U_{\alpha}(s,t)$ with respect to $Q$.
That is given a sequence of transition functions $(Q_{j})_{j \in \N}$, we are
interested in conditions such that $U_{\alpha,j}(s,t)F \longrightarrow U_{\alpha}(s,t)F$ as $j \to \infty$, where $U_{\alpha,j}(s,t)$ are the regularized
evolution systems defined as in \eqref{GJP:26}.
For functions $f \in C_b(E \times E)$ let $(Q(t)f(x,\cdot))(y) := \int _{E}f(x,w)Q(t,y,dw)$.
\begin{Lemma}\label{CONVERGENCELEMMA}
 Let $(Q_j)_{j \in \N}$ be a family of transition functions and assume that $Q_j$ is weakly continuous for any $j \in \N$. Moreover, suppose that
 the following conditions below are satisfied.
\begin{enumerate}
 \item[]{1.} Let $q_j(t,x) := Q_j(t,x,E)$, then $\sup _{\bfrac{j \geq 1}{(t,x) \in B}} q_j(t,x) < \infty$ holds for all $B \in \mathcal{C}_1$.
 \item[]{2.} For any $f \in C_b(E \times E)$ the convergence
 \begin{align}\label{GJP:08}
  (Q_j(t)f(x,\cdot))(x) \longrightarrow (Q(t)f(x,\cdot))(x), \quad  j \to \infty
 \end{align}
 is uniform in $(t,x) \in B$ for any $B \in \mathcal{C}_1$.
\end{enumerate}
Then for any $0 \leq s \leq t$ and $F \in C_b(E)$
\begin{align}\label{GJP:09}
 U_{\alpha,j}(s,t)F \longrightarrow U_{\alpha}(s,t)F, \quad j \to \infty
\end{align}
holds uniformly on compacts. If instead of \eqref{GJP:08} the stronger convergence in the total variation norm holds, i.e.
\[
 \sup _{(t,x) \in [0,T] \times B} \Vert Q_j(t,x,\cdot) - Q(t,x,\cdot)\Vert \to 0, \quad j \to \infty
\]
for any $T > 0$, then the convergence \eqref{GJP:09} is uniform on any $A \in \mathcal{C}_1$ and on $\Vert F \Vert_{\infty} \leq 1$.
\end{Lemma}
\begin{proof}
 Since $Q, Q_j$ are transition functions, it follows that $U_{\alpha,j}(s,t)$ and $U_{\alpha}(s,t)$ are Feller evolution systems on $C_b(E)$ obtained by
 \[
  U_{\alpha,j}(s,t)F(x) = \sum _{n=0}^{\infty}\alpha^n U_{j}^{(n)}(s,t)F(x)
 \]
 and
 \[
  U_{\alpha}(s,t)F(x) = \sum _{n=0}^{\infty}\alpha^n U^{(n)}(s,t)F(x).
 \]
 Since $|U^{(n)}(s,t)F(x)|, |U_{j}^{(n)}(s,t)F(x)| \leq \Vert F\Vert_{\infty}$ the convergence of the series is also uniform in $j \geq 1$.
 As a consequence it is enough to show for any $0 \leq s \leq t$, any compact $B \subset E$, $n \geq 0$ and $F \in C_b(E)$
 \begin{align}\label{GJP:10}
  \lim_{j \to \infty}\sup_{x \in B}\ |U_{j}^{(n)}(s,t)F(x) - U^{(n)}(s,t)F(x)| = 0.
 \end{align}
 For $n=0$ this follows from \eqref{GJP:08} and
 \[
  |U_{j}^{(0)}(s,t)F(x) - U^{(0)}(s,t)F(x)| \leq \Vert F \Vert_{\infty} \int _{s}^{t}|q_j(r,x) - q(r,x)|dr.
 \]
 Assume that \eqref{GJP:10} holds for one $n \geq 0$,
 proceeding by induction we obtain for $x \in B, 0 \leq s \leq t$ and $F \in C_b(E)$
 \[
  |U_j^{(n+1)}(s,t)F(x) - U^{(n+1)}(s,t)F(x)| \leq I_1 + I_2 + I_3\ ,
 \]
 where we have used \eqref{GJP:04} and
 \begin{align*}
  I_1 &= \int _{s}^{t}\int _E \Big| e^{-\int _{s}^{r}q_j(\tau,x)d\tau}
  - e^{-\int _{s}^{r}q(\tau,x)d\tau}\Big||U_j^{(n)}(r,t)F(y)|Q_j(r,x,dy)\,dr
  \\ I_2 &= \Big| \int _{s}^{t}\int _E (U_j^{(n)}(r,t)F(y)
  - U^{(n)}(r,t)F(y))e^{-\int _{s}^{r}q(\tau,x)d\tau}Q_j(r,x,dy)\,dr\Big|
  \\ I_3 &= \Big| \int _{s}^{t}\int _E e^{-\int _{s}^{r}q(\tau,x)d\tau}U^{(n)}(r,t)F(y)(Q_j(r,x,dy)
  - Q(r,x,dy))\,dr\Big|.
 \end{align*}
 The first integral can be estimated by using $|U_{j}^{(n)}(r,t)F(y)| \leq \Vert F \Vert_{\infty}$ and
 $q_j(r,x) \leq q^* := \sup _{j \geq 1}\sup _{(\tau,x) \in [s,t] \times B}\ Q_j(\tau,x,E)$ for each $r \in [s,t]$, which yields
 \begin{align*}
  I_1 &\leq \Vert F \Vert_{\infty} \int _{s}^{t}q_j(r,x)\Big| \int _{s}^{r}(q_j(\tau,x) - q(\tau,x))
  \,d\tau\Big| dr
  \\ &\leq \Vert F \Vert_{\infty} (t-s)^2 q^*\sup _{(\tau,x) \in [s,t]\times B} |q_j(\tau,x) - q(\tau,x)|.
 \end{align*}
 To estimate $I_2$ we need the following lemma.
 \begin{Lemma}
  For any $\e > 0$, $T > 0$ there exists a compact $A \subset E$ and $j_0 \geq 1$ such that
  \[
   \int _{0}^{T}Q_j(t,x,A^c)\,dt \leq \e, \quad x \in B, \quad j \geq j_0.
  \]
 \end{Lemma}

 \begin{proof}
  Since $Q$ has the localization property we can find a compact $A_1 \subset E$ such that
  \[
   \int _{0}^{T}Q(t,x,A_1^c)\,dt \leq \frac{\e}{2}, \quad x \in B.
  \]
  Choose compacts $A, A_2 \subset E$ such that $A_1 \subset \overset{\circ}{A_2} \subset A_2 \subset \overset{\circ}{A}
  \subset A$, since $(\overset{\circ}{A_2})^c$
  and  \break   $(\overset{\circ}{A})^c$ are closed there exists a continuous function $\varphi$ with $\one_{(\overset{\circ}{A})^c} \leq \varphi \leq \one_{(\overset{\circ}{A_2})^c}$.
  We obtain
  \[
   \int _{0}^{T}Q_j(t,x,A^c)\,dt \leq \int _{0}^{T}Q_j(t,x,(\overset{\circ}{A})^c)\,dt \leq
   \int _0^T \int _{E}\varphi(y)Q_j(t,x,dy)\,dt
  \]
  and by \eqref{GJP:08} there exists $j_0 \geq 1$ such that for $j \geq j_0$, $x \in B$ and $t \in [0,T]$
  \[
   \int _{E}\varphi(y) Q_j(t,x,dy) \leq \frac{\e}{2T} + \int _{E}\varphi(y)Q(t,x,dy).
  \]
  Therefore the assertion follows from
  \begin{align*}
   \int _{0}^{T}\int _{E}\varphi(y)Q_j(t,x,dy)\,dt &\leq \frac{\e}{2} + \int _{0}^{T}
   \int _{E}\varphi(y)Q(t,x,dy)\,dt
   \\ &\leq \frac{\e}{2} + \int _{0}^{T}Q(t,x,(\overset{\circ}{A_2})^c)\,dt
   \leq \frac{\e}{2} + \int _{0}^{T}Q(t,x,A_1^c)\,dt \leq \e.
  \end{align*}
 \end{proof}
 Take $A \subset E$ and $j_0 \geq 1$ as in above lemma, then for any $j \geq j_0$ and $x \in B$
 \begin{align*}
  I_2 &\leq \int _{s}^{t} \int _A |U_j^{(n)}(r,t)F(y) - U^{(n)}(r,t)F(y)|Q_j(r,x,dy)\,dr
  + 2\Vert F \Vert_{\infty} \int _{s}^{t}Q_j(r,x,A^c)\,dr
  \\ &\leq q^*\int _{s}^{t}\underset{y \in A}{\sup}\ |U_j^{(n)}(r,t)F(y) - U^{(n)}(r,t)F(y)|\,dr
  + 2 \Vert F \Vert_{\infty}  \int _{s}^{t}Q_j(r,x,A^c)\,dr
  \\ &\leq q^*\int _{s}^{t}\underset{y \in A}{\sup}\ |U_j^{(n)}(r,t)F(y) - U^{(n)}(r,t)F(y)|\,dr + 2 \Vert F \Vert_{\infty}\e.
 \end{align*}
 The integrand tends for each fixed $r \in [s,t]$ to zero as $j \to \infty$ and since
 \[
  \underset{y \in A}{\sup}\ |U_j^{(n)}(r,t)F(y) - U^{(n)}(r,t)F(y)| \leq 2 \Vert F \Vert_{\infty}
 \]
 also the integral tends to zero. Altogether this shows the assertion for $I_2$.
 For the last integral observe that $(r,x,y) \longmapsto e^{-\int _{s}^{r}q(\tau,x)d\tau}U^{(n)}(r,t)F(y)$ is continuous and moreover bounded by $\Vert F \Vert_{\infty}$.
 Therefore by \eqref{GJP:08} for any $r \in [s,t]$
 \[
  F_j(r,s,t) := \sup _{x \in B}\Big| \int _{E}e^{-\int _{s}^{r}q(\tau,x)d\tau}U^{(n)}(r,t)
  F(y)(Q_j(r,x,dy) - Q(r,x,dy))\Big| \to 0, \quad j \to \infty
 \]
 and since $F_j(r,s,t) \leq 2 \Vert F \Vert_{\infty} q^*$ we obtain the assertion by dominated convergence.
 The second assertion can be proved very similarly, here only $I_3$ should be estimated again.
\end{proof}

As a consequence we can show that $U_{\alpha}(s,t)$ satisfies a Chernoff product formula. That is $U_{\alpha}(s,t)$
can be approximated by evolution systems $U_{\alpha,n}(s,t)$ with piecewise constant (in time)
transition functions $Q_{n}$. More precisely, take for any $n \in \N$
a sequence $0 = t_{0}^{(n)} \leq t_{k}^{(n)} < t_{k+1}^{(n)}$ with $\sup _{k \geq 0}\ (t_{k+1}^{(n)} - t_{k}^{(n)}) \to 0$ as $n \to \infty$
and $t_k^{(n)} \longrightarrow \infty$, $k\to \infty$ for all $n \in \N$. Define piecewise constant transition functions by
\[
 Q_{n}(t,x,dy) = Q(t_k^{(n)},x,dy), \quad t_k^{(n)} \leq t < t_{k+1}^{(n)},\quad k \geq 0,
\]
then $Q_n$ is weakly continuous in $x$ for any fixed $t \geq 0$ and $n \geq 1$.
Denote by $U_{\alpha,n}(s,t)$ the regularized Feller evolutions on $C_b(E)$ constructed above, cf. Theorem \ref{GJPTH:00}.
For fixed $r \geq 0$ set $Q^r(x,dy) := Q(r,x,dy)$, then $Q^r$ is a weakly continuous transition function and its associated regularized
Feller evolution on $C_b(E)$ can be represented by a Feller semigroup $T_{\alpha,r}(t)$.
\begin{Lemma}\label{GJPTH:04}
 Let $F \in C_b(E)$ and for $0 \leq s \leq t$ choose $m_0, m_1 \geq 1$ such that
 \begin{align}\label{GJP:14}
  t_{m_0}^{(n)} \leq s < t_{m_0+1}^{(n)} < \cdots t_{m_1}^{(n)} \leq t < t_{m_1 +1}^{(n)}.
 \end{align}
 Then
 \begin{align}\label{GJP:15}
  U_{\alpha,n}(s,t)F(x) = T_{\alpha,t_{m_0}^{(n)}}(t_{m_0 + 1}^{(n)} -s)\cdots T_{\alpha,t_{m_1}^{(n)}}(t-t_{m_1}^{(n)})
  F \to U_{\alpha}(s,t)F, \quad n \to \infty
 \end{align}
 holds uniformly on compacts.
\end{Lemma}

\begin{proof}
 First observe that for any compact $B \subset E$, $T > 0$ and $f \in C_b(E \times E)$ Lemma~\ref{LEMMAFELLER2} implies that
 $F(r,x) := \int _{E}f(x,y)Q(r,x,dy)$ is continuous. Therefore
 \[
  \int _{E}f(x,y)Q_n(r,x,dy) \longrightarrow \int _{E}f(x,y)Q(r,x, dy), \quad n \to \infty
 \]
 holds uniformly in $(x,r) \in B \times [0,T]$ and hence \eqref{GJP:08} follows. Applying Lemma \ref{CONVERGENCELEMMA} we obtain for all $F \in C_b(E)$ and
 $0 \leq s \leq t$: $U_{\alpha,n}(s,t)F \to U_{\alpha}(s,t)F$ as $n \to \infty$ uniformly on compacts.
 By the evolution system property it follows that
 \[
  U_{\alpha,n}(s,t) = U_{\alpha}(s,t_{m_0+1}^{(n)})\cdots U_{\alpha}(t_{m_1}^{(n)}, t)
 \]
 holds. For each pair $r_0 < r_1$ with $t_{m}^{(n)} \leq r_0 < r_1 \leq t_{m+1}^{(n)}$ for some $m \geq 1$, by \eqref{GJP:03} and \eqref{GJP:04} it follows that
 $U_{\alpha,n}(r_0,r_1) = T_{\alpha,t_{m}^{(n)}}(r_1 -r_0)$ and hence
 \[
  U_{\alpha,n}(s,t) = T_{\alpha,t_{m_0}^{(n)}}(t_{m_0+1}^{(n)}-s)\cdots T_{\alpha,t_{m_1}^{(n)}}(t - t_{m_1}^{(n)})
 \]
 implies the assertion.
\end{proof}
In the following we consider the limit $\alpha \to 1$ and deduce from that $U(s,t)1 = 1$.
\begin{Theorem}
 The evolution system $U(s,t)$ is conservative and can be extended to $BM_V(E)$ so that
 \begin{align*}
 \Vert U(s,t)F\Vert_V \leq e^{\int _{s}^{t}c(r)dr} \Vert F \Vert_V, \quad 0 \leq s \leq t.
 \end{align*}
\end{Theorem}

\begin{proof}
 Denote by $T_{\alpha,r}(t)$ the regularized semigroups with piecewise constant (in the time variable) transition functions, see Theorem \ref{GJPTH:04}
 and by $T_{r}(t)$ their counterparts with $\alpha=1$. Then $T_{\alpha,r}(t)V(x) \leq T_r(t)V(x)$.
 The moment condition \eqref{GJP:12} and the results obtained in  \cite{MU-FA04, KOL06} imply that for any $r \geq 0$, $x \in E$ and $t \geq 0$
 \[
  T_{r}(t)V(x) \leq e^{c(r)t}V(x).
 \]
 Now given $0 \leq s < t$ and $n \in \N$ we can find $m_0,m_1 \geq 0$ with \eqref{GJP:14}. For $m \geq 0$ let $V_m(x) := V(x)\wedge m$,
 then $V_m \in C_b(E)$ and hence
 \begin{align*}
  U_{\alpha,n}(s,t)V_m(x) &= T_{\alpha,t_{m_0}^{(n)}}(t_{m_0+1}^{(n)}-s)\cdots T_{\alpha,t_{m_1}^{(n)}}(t - t_{m_1}^{(n)})V_m(x)
  \\ &\leq T_{t_{m_0}^{(n)}}(t_{m_0+1}^{(n)}-s)\cdots T_{t_{m_1}^{(n)}}(t - t_{m_1}^{(n)})V(x)
  \\ &\leq V(x)\exp \left( c(t_{m_1}^{(n)})(t -t_{m_1}^{(n)}) + \dots + c(t_{m_0}^{(n)})(t_{m_0+1}^{(n)}-s)\right).
 \end{align*}
 Letting $n\to \infty$ yields
 \[
  U_{\alpha}(s,t)V_m(x) \leq V(x)\exp\Big( \int _{s}^{t}c(r)\,dr\Big).
 \]
 The sequence $(U_{\alpha}(s,t)V_m(x))_{m \in \N}$ is increasing and bounded, so by monotone convergence it follows that
 \begin{align}\label{GJP:02}
  \int _{E}V(x)P_{\alpha}(s,x;t,dy) \leq V(x)\exp\Big( \int _{s}^{t}c(r)\,dr\Big)
 \end{align}
 is satisfied.
 The right-hand side is increasing in $\alpha$, hence taking the limit $\alpha \to 1$ yields that $U(s,t)$ can be extended to $BM_V(E)$.
 By \cite{FELLER40} the evolution system $U(s,t)$ is conservative if and only if for any $s < t$, $x \in E$
 \[
  \int _{s}^{t}\int _{E}q(r,y)P^{(n)}(s,x;r,dy)\,dr \longrightarrow 0, \quad n \to \infty.
 \]
 Let $T > 0$ such that $[s,t] \subset [0,T]$, then by $q(r,y) \leq a(T)V(y)$ for $r \in [s,t]$ and \eqref{GJP:02} with $\alpha = 1$
 \[
  \int _{s}^{t}\int _{E}q(r,y)P(s,x;r,dy)\,dr \leq a(T)V(x) \int _{s}^{t} \exp\Big( \int
  _{s}^{r}c(\tau)\,d\tau\Big)dr < \infty
 \]
 follows.
 The assertion follows from the representation $P(s,x;r,dy) = \sum _{n=0}^{\infty}P^{(n)}(s,x;r,dy)$.
\end{proof}
The next result shows that $U(s,t)$ is differentiable in $s$.
\begin{Theorem}\label{GJPTH:02}
 For any $F \in BM(E)$, $t > 0$ and $x \in E$, $[0,t] \ni s \longmapsto U(s,t)F(x)$ is continuously differentiable and a solution to
   \begin{align*}
    \frac{\partial}{\partial s}U(s,t)F(x) = - L(s)U(s,t)F(x).
   \end{align*}
  Moreover, for any $F \in C_V(E)$ the function $U(s,t)F(x)$ is absolutely continuous in $s$ and solves above equation a.e..
  Let $V(s,t)$ be a Feller evolution system on $C_b(E)$ and assume that $V(s,t)F$ is a solution to \eqref{GJP:16} for any $F \in C_b(E)$, then $V(s,t) = U(s,t)$
  is fulfilled.
\end{Theorem}

\begin{proof}
 By \eqref{GJP:00} we obtain for any $A \in \mathcal{B}(E)$ and $0 \leq s \leq t$
 \[
  P(s,x;t,A) = \delta(x,A) + \int _{s}^{t}q(r,x)P(r,x;t,A)\,dr - \int _{s}^{t}\int
  _{E}P(r,y;t,A)Q(r,x,dy)\,dr
 \]
 and hence for any $F \in BM(E)$ and $x \in E$
 \[
  U(s,t)F(x) = F(x) + \int _{s}^{t}q(r,x)U(r,t)F(x)\,dr - \int _{s}^{t}\int _{E}U(r,t)F(y)Q(r,x,dy)\,dr
 \]
 follows.
 Clearly $q(r,x)U(r,t)F(x)$ and by Lemma \ref{LEMMAFELLER2} also $\int _{E}U(r,t)F(y)Q(r,x,dy)$ are continuous in $r$,
 which implies that $L(r)U(r,t)F(x)$ is continuous in $(r,t)$. Therefore
 \begin{align}\label{GJP:17}
  U(s,t)F(x) = F(x) - \int _{s}^{t}L(r)U(r,t)F(x)\,dr
 \end{align}
 implies \eqref{GJP:16}. If $F \in C_V(E)$, then $U(s,t)F(x)$ is bounded and measurable in $(s,t)$.
 Hence by \eqref{GJP:12} $L(r)U(r,t)F(x)$ is well-defined and integrable w.r.t. $r$. In view of \eqref{GJP:17}
 it follows that $s \longmapsto U(s,t)F(x)$ is absolutely continuous and satisfies \eqref{GJP:16} for any $x \in E$.

Now let $V(s,t)$ be a Feller evolution on $C_b(E)$ which satisfies
\eqref{GJP:16}. By \cite[Chapter 2, Theorem 2.9]{VANCASTEREN11}
$V(s,t)$ is given by
 \[
  V(s,t)F(x) = \int _{E}F(y)\widetilde{P}(s,x;t,dy), \quad x \in E, \quad 0 \leq s \leq t,
 \]
 where $\widetilde{P}$ is a transition probability function. Moreover, this evolution system satisfies \eqref{GJP:17} for any $F \in C_b(E)$ and hence by
 approximation also for any $F \in BM(E)$. Therefore for any $F = \one_A$, $A \in \mathcal{B}(E)$ it solves equation \eqref{GJP:16} which is simply \eqref{GJP:00}.
 The minimality of $P$ implies $P \leq \widetilde{P}$ and hence $U(s,t)F \leq V(s,t)F$.
 Since $U(s,t)$ is conservative it follows that $P(s,x;t,dy)$ is the unique solution to \eqref{GJP:00}, i.e. $P(s,x;t,dy) = \widetilde{P}(s,x;t,dy)$.
\end{proof}

\begin{Theorem}
 Let $F \in BM(E)$, then for any $x \in E$ and $s \geq 0$, $[s, \infty) \ni t \longmapsto U(s,t)F(x)$ is absolutely continuous and satisfies for a.a. $t \geq s$
  \begin{align*}
   \frac{\partial}{\partial t}U(s,t)F(x) = U(s,t)L(t)F(x).
  \end{align*}
  Let $V(s,t)$ be a Feller evolution system on $C_b(E)$ and assume that $V(s,t)F$ is for any $F \in C_b(E)$ a solution to \eqref{GJP:20}, then $V(s,t) = U(s,t)$ holds.
\end{Theorem}
\begin{proof}
 For all $0 \leq s \leq r \leq t < T$
 \begin{align}\label{GJP:18}
  \int _{E}q(r,y)P(s,x;t,dy) \leq a(T)\int _{E}V(y)P(s,x;t,dy) \leq a(T)V(x)e^{\int _{s}^{t}c(r)dr}
 \end{align}
 and \eqref{GJP:01} implies for any $0 \leq s \leq t$ and compact $A \subset E$
 \[
  P(s,x;t,A) = \delta(x,A) - \int _{s}^{t}\int _{A}q(r,y)P(s,x;r,dy)\,dr + \int _{s}^{t}
  \int _{E}Q(r,y,A)P(s,x;r,dy)\,dr.
 \]
 By \eqref{GJP:18} this implies
 \[
  U(s,t)F(x) = F(x) - \int _{s}^{t}\int _{E}q(r,y)F(y)P(s,x;r,dy)\,dr + \int _{s}^{t}
  \int _{E}Q(r)F(y)P(s,x;r,dy)\,dr
 \]
 and hence
 \begin{align*}
  U(s,t)F(x) = F(x) + \int _{s}^{t}\int _{E}L(r)F(y)P(s,x;r,dy)\,dr
 \end{align*}
 holds. The first assertion is proved. Uniqueness follows by the same arguments as for \eqref{GJP:16}.
\end{proof}

\begin{Remark}\label{GJPREMARK:02}
 It is worth noting that in the time-homogeneous case \eqref{GJP:16} and \eqref{GJP:20} are equivalent and less restrictive conditions are sufficient to
 show that $U(s,t)$ is an Feller evolution, see \cite{KOL06}.
\end{Remark}
Since $U(s,t)$ is given by a transition probability function we see that
for each $x \in E$ and $s \geq 0$ there exists a probability space $(\Omega, \mathcal{F}^{s}, \mathbb{P}_{s,x})$
and a conservative Markov process $(X(t))_{t \geq s}$ on this space such that
\[
 U(s,t)F(x) = \E_{s,x}(F(X(t)), \quad F \in C_b(E), \quad t \geq s.
\]
This process is considered w.r.t. its natural filtration defined by $\mathcal{F}_{\tau}^s = \sigma\left( X(t) \ | \ s \leq t \leq \tau \right)$
for $s \leq \tau$. Note that this process is, by construction, a pure jump process.
The next statement completes the proof of Proposition \ref{GJPTH:01}.
\begin{Corollary}\label{GJPTH:05}
 The following statements are true:
 \begin{enumerate}
  \item[]{1.} Let $F \in BM(E)$. Then for any fixed $s \geq 0$
  \[
   M_{s,F}(t) := F(X(t)) - F(X(s)) - \int _{s}^{t}L(r)F(X(r))\,dr, \quad t \geq s
  \]
  is a martingale with respect to $(\mathcal{F}_t^{s})_{t \geq s}$ and $\mathbb{P}_{s,x}$.
  \item[]{2.} For any $a > 0$, $x \in E$ and $0 \leq s < T$
  \[
   \mathbb{P}_{s,x}\Big( \sup _{t \in [s,T]} V(X(t)) \geq a \Big) \leq V(x) \frac{e^{\int _{s}^{T}c(r)dr}}{a}
  \]
  holds.
  \item[]{3.} $U(s,t)$ is a Feller evolution system.
 \end{enumerate}
\end{Corollary}
\begin{proof}
 1. This follows by the Markov property and the relation
 \[
  \E_{\tau,X(\tau)}(L(r)F(X(r))) = (U(\tau,r)L(r)F)(X(\tau)) = \frac{\partial}{\partial r}U(\tau,r)F(X(\tau)),
 \]
 where $0 \leq s \leq \tau \leq t$.
 \\2. Let $E_n := \{ x \in E \ | \ V(x) < n \}$, fix $s \geq 0$ and define a family of stopping times
 \[
  \tau_n := \inf \{ t \geq s \ | \ X_t \not \in E_n \}.
 \]
 Let $\varphi_n \in C(E)$ be such that $\one_{\overline{E}_n} \leq \varphi_n \leq \one_{\overline{E}_{n+1}}$
 and define a new transition function by $Q_n(t,x,dy) := \varphi_n(x)Q(t,x,dy)$.
 Then
 \[
  L_n(t)F(x) := -\varphi_n(x)q(t,x)F(x) + \int _{E}F(y)\varphi_n(x)Q(t,x,dy) = \varphi_n(x)L(t)F(x)
 \]
 determines a bounded linear operator on $C_b(E)$ and $BM_V(E)$.
 Hence there exists an associated conservative Feller evolution system $U_n(s,t)$ on $C_b(E)$. This evolution system can be extended to $BM_V(E)$.
 Let $(X_t^n)_{t \geq 0}$ be the corresponding Markov process, and denote by $(\mathcal{F}_{t,n}^{s})_{t \geq s}$ its associated natural filtration.
 By construction it follows for $x \in E_n$ and $n \geq 1$ that these processes satisfy
 \begin{align}\label{GJP:25}
  (X_t)_{t < \tau_n} = (X_t^n)_{t < \tau_n}
 \end{align}
 in the sense of finite dimensional distributions. For $s \geq 0$ let $g(t,x):= e^{-\int _{s}^{t}c(r)dr}V(x)$.
 A short computation shows that
 \[
  \frac{\partial}{\partial t}g(t,x) + L(t)g(t,x) \leq 0.
 \]
 Then
 \[
  M_n(s,t) := g(t,X^n(t)) - g(s,X^n(s)) - \int _{s}^{t}\left(\frac{\partial }{\partial r} + L_n(r)\right)
  g(r,X^n(r))\,dr, \quad t \geq s
 \]
 is a $\mathcal{F}_{t,n}^s$-martingale w.r.t. $\mathbb{P}_{s,x}$. Fix $x \in E_n$, $n \geq 1$, hence by Dynkin's formula
 \begin{align}\label{GJP:21}
  \E_{s,x}(g( t \wedge \tau_n, X_{t \wedge \tau_n}^n)) &= g(s,x) + \E_{s,x}\Bigg( \int _{s}^{t
  \wedge \tau_n}\left( \frac{\partial }{\partial r} + L_n(r)\right)g(r,X_r^n)\,dr\Bigg)
  \leq g(s,x)
 \end{align}
 holds. Here $\frac{\partial}{\partial r}$ acts only on the first variable of $g$.
 Let $M_{t}^n := e^{- \int _s^{t}c(\sigma)d\sigma} V(X_t^n)\one_{t < \tau_n}$, we will show that $(M_{t}^n)_{t \geq s}$ is a supermartingale.
 Fix $s \leq r \leq t$. On $\{ r \geq \tau_n \} \in \mathcal{F}_{r,n}^s$ we have $M_{t}^n = M_{r}^n = 0$ and hence obtain
 \[
  \E_{s,x}(M_{t}^n| \mathcal{F}_{r,n}^s) = M_{r}^n = 0.
 \]
 On $\{r < \tau_n\}$ we have by the Markov property and \eqref{GJP:21}
 \begin{align*}
  \E_{s,x}(M_{t}^n | \mathcal{F}_{r,n}^s) &= e^{-\int _{s}^{t}c(\sigma)d\sigma} \E_{r, X_r^n}(V(X_{t}^n)\one_{t < \tau_n})
  \leq \E_{r,X_r^n}(g(t \wedge \tau_n, X_{t \wedge \tau_n}^n))
  \\ &\leq g(r,X_r^n) = g(r \wedge \tau_n, X_{r \wedge \tau_n}^n) = M_{r}^n.
 \end{align*}
 Applying Doob's inequality yields
 \begin{align*}
  \mathbb{P}_{s,x}\Big(\sup _{\bfrac{s \leq t  \leq T}{t < \tau_n}}g(t,X(t)) \geq a \Big) = \mathbb{P}_{s,x}
  \Big( \sup _{s \leq t \leq T} M_t^n \geq a \Big) \leq \frac{1}{a}\E_{s,x}(M_s^n) = \frac{V(x)}{a}.
 \end{align*}
 As a consequence we obtain
 \begin{align*}
  \mathbb{P}_{s,x}\Big( \sup _{\bfrac{s \leq t \leq T}{t < \tau_n}} V(X(t)) \geq a \Big) &\leq \mathbb{P}_{s,x}
  \Big( \sup _{\bfrac{s \leq t \leq T}{t < \tau_n}} g(t,X(t)) \geq a e^{-\int _{s}^{T}c(r)dr}\Big)
  \\ &\leq V(x) \frac{e^{\int _{s}^{T}c(r)dr}}{a}.
 \end{align*}
 Since $(X_t)_{t \geq s}$ is conservative it follows $\tau_n \longrightarrow \infty$ when $n \to \infty$.
 The assertion follows by monotone convergence and $n \to \infty$.
 \\ 3. For any $F \in C_b(E)$, $x \in E_n$ and $n \geq 1$ it follows by \eqref{GJP:25}
 \begin{align*}
  | \E_{s,x}(F(X_t)) - \E_{s,x}(F(X_t^n))| &=  | \E_{s,x}(F(X_t) \one_{\tau_n \leq t}) - \E_{s,x}(F(X_t^n)\one_{\tau_n \leq t})|
  \\ &\leq 2 \Vert F \Vert_{\infty}\mathbb{P}_{s,x}(\tau_n \leq t).
 \end{align*}
 By
 \begin{align*}
  \mathbb{P}_{s,x}(\tau_n \leq t) &\leq \mathbb{P}_{s,x}\left( \sup _{s \leq r \leq t} V(X(r)) \geq n \right)
   \leq \frac{V(x)}{n}e^{\int _{s}^{t}c(r)dr}.
 \end{align*}
 and the continuity of $V$ we see that $U_n(s,t)F(x) \longrightarrow U(s,t)F(x)$ uniformly on compacts which implies the assertion.
\end{proof}
We close this section with the relation to the evolution of measures. Let $\mathcal{M}(E)$ be the space of all finite, signed Borel measures on $E$ equipped with the
total variation norm. Define bounded linear operators $(U^*(t,s))_{0 \leq s \leq t}$ on $\mathcal{M}(E)$ by
\[
 U^*(t,s)\mu(dx) = \int _{E}P(s,y;t,dx) \mu(dy).
\]
Then $U^*(t,t) = \mathrm{id}_{\mathcal{M}(E)}$, $U^*(t,r)U^*(r,s) = U^*(t,s)$ holds for $0 \leq s \leq r \leq t$ and
\[
 \int _{E}F(y)U^*(t,s)\mu(dy) = \int _{E}U(s,t)F(y)\mu(dy), \quad F \in C_b(E), \quad \mu \in \mathcal{M}(E).
\]
Previous considerations show that $U^*(t,s)\mu$ is the unique weak solution to the Fokker-Planck equation
\[
 \frac{\partial}{\partial t} \int _{E}F(y)U^*(t,s)\mu(dy) = \int _{E}L(t)F(y)U^*(t,s)\mu(dy),
\]
where $\mu \in \mathcal{M}(E)$ is such that $\int _{E}V(x)|\mu|(dx) < \infty$.

\section{Interacting particle systems in continuum}

\subsection*{Preliminaries}
The configuration space $\Gamma_0$ is the space of all finite subsets of $\R^d$, i.e.
\[
 \Gamma_0 = \{ \eta \subset \R^d\ | \ |\eta| < \infty \},
\]
where $|\eta|$ denotes the number of elements in the set $\eta$. This space has a natural decomposition into $n$-particle spaces,
$\Gamma_0 = \bigsqcup _{n=0}^{\infty}\Gamma_0^{(n)}$, where $\Gamma_0^{(n)} = \{ \eta \subset \R^d\ | \ |\eta| = n\}, \ \ n \geq 1$
and in the case $n = 0$ we set $\Gamma_0^{(0)} = \{\emptyset\}$. For a compact $\Lambda \subset \R^d$ let
\[
 \Gamma_{\Lambda} = \{ \eta \in \Gamma_0\ | \ \eta \subset \Lambda\}
\]
and $\Gamma_{\Lambda}^{(n)} = \{ \eta \in \Gamma_0^{(n)}\ | \ \eta \subset \Lambda\}$.
Denote by $\widetilde{(\R^d)^n}$ the space of all sequences $(x_1, \dots, x_n) \in (\R^d)^n$ with $x_i \neq x_j$ for $i \neq j$.
$\Gamma_0^{(n)}$ can be identified with $\widetilde{(\R^d)^n}$ via the symmetrization map
\[
 \mathrm{sym}_n: \widetilde{(\R^d)^n} \longrightarrow \Gamma_0^{(n)}, \ (x_1, \dots, x_n) \longmapsto \{x_1, \dots, x_n\},
\]
which defines a topology on $\Gamma_0^{(n)}$.
Namely, a set $A \subset \Gamma_0^{(n)}$ is open if and only if $\mathrm{sym}_n^{-1}(A) \subset \widetilde{(\R^d)^n}$ is open.
On $\Gamma_0$ we define the topology of disjoint unions, i.e. a set $A \subset \Gamma_0$ is open iff
$A \cap \Gamma_0^{(n)}$ is open in $\Gamma_0^{(n)}$ for all $n \in \N$. Then $\Gamma_0$ is a locally compact Polish space. Let $\mathcal{B}(\Gamma_0)$
stand for the Borel-$\sigma$-algebra on $\Gamma_0$. With respect to this topology the function
\[
 \eta \longmapsto \langle f, \eta \rangle := \sum _{x \in \eta}f(x)
\]
is continuous whenever $f \in C_b(\R^d)$. Therefore convergence of a sequence $(\eta_n)_{n \in \N} \subset \Gamma_0$ to $\eta \in \Gamma_0$ can be rewritten to:
there exists $N \in \N$ such that for all $n \geq N$: $\eta_n = \{ x_1^{(n)}, \dots, x_l^{(n)}\}, \ \eta = \{x_1, \dots, x_l\}$ and
\[
 x_j^{(n)} \longrightarrow x_j, \ \ n \to \infty, \ \forall j \in \{1, \dots, l\}
\]
is fulfilled. For given $\delta > 0$, $N \in \N_0$ and a compact $\Lambda \subset \R^d$ the set
\begin{align}\label{COMPACT}
 B = \{ \eta \in \Gamma_{\Lambda}\ | \ \forall x \neq y, \ x,y \in \eta: \ |x-y| \geq \delta, \ \ |\eta| \leq N\}
\end{align}
is compact. Conversely, for any compact set $A \subset \Gamma_0$ there exist $\delta, N, \Lambda$ such that $A$ is contained in a compact $B$ defined above.
Denote by $dx$ the Lebesgue measure on $\R^d$ and by $d^{\otimes n} x$ the product measure on $(\R^d)^n$.
The image measure of $d^{\otimes n}x$ on $\Gamma_0^{(n)}$ via $\mathrm{sym}_n$ is then denoted by $d^{(n)}x$. The Lebesgue-Poisson measure is defined by
\[
 \lambda = \delta_{\emptyset} + \sum _{n=1}^{\infty}\frac{1}{n!}\,d^{(n)}x.
\]
Given a measurable function $G:\Gamma_0 \times \Gamma_0 \longrightarrow \R$, then
\begin{align}\label{PHD:00}
 \int _{\Gamma_0}\sum _{\xi \subset \eta}G(\xi, \eta \backslash \xi)\,d\lambda(\eta)
 = \int _{\Gamma_0}\int _{\Gamma_0}G(\xi, \eta)d\lambda(\xi)\,d\lambda(\eta)
\end{align}
holds, provided one side of the equality is finite for $|G|$. Here and in the following we write $\eta \backslash x$, $\eta \cup x$, instead of $\eta \backslash \{x\}$
and $\eta \cup \{x\}$. The decomposition $\Gamma_0 = \bigsqcup _{n=0}^{\infty}\Gamma_0^{(n)}$ implies that any measurable
function $G: \Gamma_0 \longrightarrow \R$ can be represented as a sequence of symmetric measurable functions $(G^{(n)})_{n=0}^{\infty}$, where
$G^{(n)}: (\R^d)^n \longrightarrow \R$. Such functions are uniquely determined on the off-diagonal part $\widetilde{(\R^d)^n}$ and integration w.r.t. to the Lebesgue-Poisson
measure is simply determined by the identity
\[
 \int _{\Gamma_0}G(\eta)\,d\lambda(\eta) = G^{(0)} + \sum _{n=1}^{\infty}\frac{1}{n!}\int
 _{(\R^d)^n}G^{(n)}(x_1, \cdots, x_n)\,dx_1 \cdots dx_n.
\]
For a given measurable function $f: \R^d \longrightarrow \R$ the Lebesgue-Poisson exponential is defined by
\[
 e_{\lambda}(f;\eta) := \prod _{x \in \eta}f(x)
\]
and satisfies the combinatorial formula
\[
 \sum _{\xi \subset \eta}e_{\lambda}(f;\xi) = e_{\lambda}(1+f;\eta).
\]
For computations we will use the identity
\[
 \int _{\Gamma_0}e_{\lambda}(f;\eta)\,d\lambda(\eta) = \exp\Big( \int _{\R^d}f(x)\,dx\Big),
\]
whenever $f \in L^1(\R^d)$.

\subsection*{General statement}
The class of Markov jump processes we are interested in are given by a Markov pre-generator of the form
\begin{align}\label{IPSFINITE:01}
 (L(t)F)(\eta) = \sum _{\xi \subset \eta}\int _{\Gamma_0}(F(\eta \backslash \xi \cup \zeta)
 - F(\eta))K_t(\xi, \eta, d\zeta), \quad \eta \in \Gamma_0, \quad t \geq 0.
\end{align}
Such Kolmogorov operator includes death, birth and jumps of groups of particles. We will say $K_t$ satisfies the usual conditions if the conditions given below are satisfied.
\begin{enumerate}
 \item[]{1.} For all $\eta, \xi \in \Gamma_0$ and $t \geq 0$: $K_t(\xi, \eta, \cdot) \geq 0$ is a finite, non-atomic Borel measure.
 \item[]{2.} For all $A \in \mathcal{B}(\Gamma_0)$, the map $(t,\xi, \eta) \longmapsto K_t(\xi, \eta, A)$ is measurable.
\end{enumerate}
For $t \geq 0$, $\eta \in \Gamma_0$ and $A \in \mathcal{B}(\Gamma_0)$ define $Q(t,\eta,d\omega)$ by
\begin{align}\label{IPSFINITE:00}
 Q(t,\eta,A) = \sum _{\xi \subset \eta} \int _{\Gamma_0}\one_{A}(\eta \backslash \xi \cup \zeta)K_t(\xi, \eta, d\zeta).
\end{align}
The cumulative intensity is defined by $q(t,\eta) := Q(t,\eta, \Gamma_0) = \sum _{\xi \subset \eta}K_t(\xi, \eta, \Gamma_0)$.
We will work with the following conditions:
\begin{enumerate}
 \item[(A)] For any $\e > 0,\ T > 0$ and any compact $B \subset \Gamma_0$ there exists another compact $A \subset \Gamma_0$ such that
  \[
   \int _{0}^{T}Q(r,\eta,A^c)\,dr < \e, \quad \eta \in B
  \]
  is satisfied.
 \item[(B)] There exist continuous functions $V:\Gamma_0 \longrightarrow \R_+$ and $c:\R_+ \longrightarrow \R_+$ such that
  \begin{align}\label{PHD:08}
   \sum _{\xi \subset \eta}\int _{\Gamma_0}V(\eta \backslash \xi \cup \zeta)K_t(\xi,\eta,d\zeta)
   \leq c(t)V(\eta) + q(t,\eta)V(\eta), \quad t \geq 0, \quad \eta \in \Gamma_0
  \end{align}
  holds.
 \item[(C)] For any $F \in C(\Gamma_0)$ with $\sup _{\eta \in \Gamma_0}\frac{|F(\eta)|}{1+V(\eta)} < \infty$
 \[
  (t,\eta) \longmapsto \sum _{\xi \subset \eta}\int _{\Gamma_0}F(\eta \backslash \xi \cup \zeta)K_t(\xi,\eta, d\zeta)
 \]
 is continuous.
 \item[(D)] For any $T > 0$ there exists $a(T) > 0$ such that $q(t,\eta) \leq a(T)V(\eta)$ holds for all $\eta \in \Gamma_0$ and $t \in [0,T]$.
 \item[(E)] For any $T > 0$ there exists $b(T) > 0$ such that $q(t,\eta) \geq b(T)q(T,\eta)$ holds for all $\eta \in \Gamma_0$ and $t \in [0,T]$.
\end{enumerate}
As in the previous section let $BM_V(\Gamma_0)$ stand for the Banach space of all measurable functions $F$ equipped with the norm
$\Vert F \Vert_V = \sup _{\eta \in \Gamma_0}\frac{|F(\eta)|}{1+V(\eta)}$. Denote by $C_V(\Gamma_0)$ the closed subspace of all continuous functions
for which $\Vert \cdot \Vert_V$ is finite. Then condition (D) simply states that for any $F \in C_V(\Gamma_0)$ the action $L(t)F$, cf. \eqref{IPSFINITE:01}, is continuous in $(t,\eta)$.
\begin{Proposition}\label{IPSFINITETH:00}
 Let $K_t$ be a transition function with the usual conditions and assume that conditions (A)--(D) hold. Then there exists a unique associated conservative Feller evolution
 $U(s,t)$ on $C_b(\Gamma_0)$. This evolution system can be extended to $BM_V(\Gamma_0)$ so that
 \begin{align}\label{PHD:25}
  |U(s,t)F(\eta)| \leq \Vert F \Vert_V V(\eta)e^{\int _{s}^{t}c(r)dr}.
 \end{align}
 Moreover the following assertions are true:
 \begin{enumerate}
  \item[]{1.} For any $F \in BM(\Gamma_0)$, $t > 0$ and $\eta \in \Gamma_0$, $U(s,t)F(\eta)$ is a solution to
  \[
   \frac{\partial}{\partial s}U(s,t)F(\eta) = - L(s)U(s,t)F(\eta), \quad s \in [0,t).
  \]
  \item[]{2.} Let $F \in BM(\Gamma_0)$. Then for any $s \geq 0$ and $\eta \in \Gamma_0$, $U(s,t)F(\eta)$ is a solution to
  \[
   \frac{\partial}{\partial t}U(s,t)F(\eta) = U(s,t)L(t)F(\eta), \quad \text{ a.a. } t \geq s.
  \]
 \end{enumerate}
\end{Proposition}
\begin{proof}
 The assertion follows by Proposition \ref{GJPTH:01}.
\end{proof}
The considerations of the first section imply that $U(s,t)F$ is given by a transition probability function $P(s,\eta;t,d\omega)$, that is
\begin{align}\label{PHD:07}
 U(s,t)F(\eta) = \int _{\Gamma_0}F(\omega)P(s,\eta;t,d\omega)
\end{align}
holds. The adjoint evolution system on $\mathcal{M}(\Gamma_0)$ is given by
\[
 U^*(t,s)\mu(A) = \int _{\Gamma_0}P(s,\eta;t,A)\,d\mu(d\eta).
\]
The action of the adjoint evolution $U^*(t,s)\mu$ provides a weak solution to the Fokker-Planck equation
\[
 \frac{\partial}{\partial t}\int _{\Gamma_0}F(\eta)\mu_t(d\eta) = \int _{\Gamma_0}L(t)F(\eta)\,d\mu_t(\eta),
 \quad F \in C_b(\Gamma_0).
\]
In particular, if conditions (A)--(D) are satisfied, then $U(t,s)^*$ is unique with such pro\-perty.

Here and in the following we identify the space of densities
$L^1(\Gamma_0,d\lambda)$ with its image in $\mathcal{M}(\Gamma_0)$ given by the (isometric) embedding
\[
 L^1(\Gamma_0, d\lambda) \ni R \longmapsto Rd\lambda \in \mathcal{M}(\Gamma_0).
\]
The next theorem states conditions for which $U^*(t,s)$ leaves the space of
densities invariant and its restriction to $L^1(\Gamma_0,d\lambda)$ is strongly continuous.
\begin{Theorem}\label{IPSFINITETH:01}
 Assume that $K_t(\xi,\eta,d\zeta)$ satisfies the usual conditions, is absolutely continuous with respect to the Lebesgue-Poisson measure
 and the conditions (A)--(E) hold. Then $U^*(t,s)$ leaves $L^1(\Gamma_0, d\lambda)$
 invariant and is strongly continuous on $L^1(\Gamma_0,d\lambda)$.
\end{Theorem}
\begin{proof}
 Denote by $K_t(\xi,\eta,\zeta) = \frac{dK_t(\xi,\eta,d\zeta)}{d\lambda(\zeta)}$ and
 let $L^*(t)$ be the adjoint operator with respect to the duality of $BM(\Gamma_0)$ and $\mathcal{M}(\Gamma_0)$.
 Then $L^*(t)$ is given by $L^*(t) = -q(t,\cdot) + Q(t)$ with $(-q(t,\cdot)R)(\eta) = -q(t,\eta)R(\eta)$ and
 \begin{align}\label{PHD:03}
  Q(t)R(\eta) = \sum _{\xi \subset \eta}\int _{\Gamma_0}R(\eta \backslash \xi \cup \zeta)
  K_t(\zeta,\eta \backslash \xi \cup \zeta,\xi)\,d\lambda(\zeta),
 \end{align}
 see \eqref{PHD:00}. For $t \geq 0$ let
 \begin{align}\label{PHD:04}
  D(L^*(t)) = \{ R \in L^1(\Gamma_0, d\lambda) \ | \ q(t,\cdot)R \in L^1(\Gamma_0, d\lambda) \}.
 \end{align}
 First observe that $W^*(t,s)R(\eta) = e^{-\int _{s}^{t}q(r,\eta)dr}R(\eta)$ is a positive contraction operator and $Q(t)$ is positive.
 In order to apply \cite[Theorem 2.1]{ALM14}
 it is enough to show that for a.a. $t > s$ and all $R \in L^1(\Gamma_0,d\lambda)$: $W^*(t,s)R \in D(L^*(t))$ and
 \[
  \int _{s}^{t}\Vert Q(r)W^*(r,s)R\Vert_{L^1}dr \leq \Vert R \Vert_{L^1} - \Vert W^*(t,s)R \Vert_{L^1}.
 \]
 The first property follows by property (E) from
 \[
  \int _{\Gamma_0}q(t,\eta)|W^*(t,s)R(\eta)|\,d\lambda(\eta) \leq \int _{\Gamma_0}q(t,\eta)e^{-b(t)(t-s)
  q(t,\eta)}|R(\eta)|\,d\lambda(\eta) \leq \frac{\Vert R \Vert_{L^1}}{b(t)(t-s)e}.
 \]
 For the second property let $R \in L^1(\Gamma_0, d\lambda)$ and note that
 \[
  \int _{\Gamma_0}|Q(r)R(\eta)| \,d\lambda(\eta) \leq \int _{\Gamma_0}q(r,\eta)|R(\eta)|\,d\lambda(\eta)
 \]
 holds. Altogether this implies
 \begin{align*}
  \int _{s}^{t}\Vert Q(r)W^*(r,s)R\Vert_{L^1}dr &\leq \int _{s}^{t}
  \int _{\Gamma_0}q(r,\eta)e^{-\int _{s}^{r}q(\tau,\eta)d\tau}|R(\eta)|\,d\lambda(\eta)dr
  \\ &= -\int _{s}^{t}\int _{\Gamma_0}\frac{\partial}{\partial r}e^{\!-\!\int _{s}^{r}
  q(\tau,\eta)d\tau}|R(\eta)|\,d\lambda(\eta)dr
  \\
 & = \Vert R \Vert_{L^1}\! -\! \Vert W^*(t,s)R\Vert_{L^1}.
 \end{align*}
 Hence by \cite[Theorem 2.1]{ALM14} there exists a strongly continuous evolution
 family \break
 $(V^*(t,s))_{0 \leq s \leq t}$ on $L^1(\Gamma_0, d\lambda)$.
 The construction of $V^*(t,s)$ coincides with the construction of $U^*(t,s)$ restricted to $L^1(\Gamma_0,d\lambda)$,
 i.e. $U^*(t,s)R = \sum _{n=0}^{\infty}U_n^*(t,s)R$ with $U_0^*(t,s)R = e^{-\int _{s}^{t}q(r,\eta)dr}R$ and
 \[
  U_{n+1}^*(t,s)R = \int _{s}^{t}U_n^*(r,t)Q(r)W^*(r,s)R\,dr,
 \]
 cf. \cite[Section 3, Theorem 1]{FELLER40}.
\end{proof}
\begin{Remark}
 For the application of \cite[Theorem 2.1]{ALM14} it is necessary to show that $t \longmapsto Q(t)R \in L^1(\Gamma_0, d\lambda)$ is measurable.
 Since $L^1(\Gamma_0, d\lambda)$ is separable, strong measurability and weak
 measurability coincide, which is the reason why we have to restrict the evolution to the space of densities.
\end{Remark}

\subsection*{Examples}
In this part we study two particular models, which have applications in ecological sciences. To simplify the proofs we consider first the case of a
Markov (pre-)generator describing only the death of particles.
\begin{Lemma}\label{DEATHLEMMA}
 Consider the operator $L(t)$ given by
 \[
  (L(t)F)(\eta) = \sum _{\xi \subset \eta}(F(\eta \backslash \xi) - F(\eta))D_t(\xi, \eta), \quad t \in I,
 \]
 where $(t,\xi, \eta) \longmapsto D_t(\xi, \eta)\geq 0$ is assumed to be continuous.
 Then condition (A) and the usual conditions holds. Moreover, $(t,\eta) \longmapsto L(t)F(\eta)$ is continuous for any $F \in C(\Gamma_0)$.
\end{Lemma}
\begin{proof}
 The associated function is given by $K_t(\xi,\eta,d\zeta) = D_t(\xi,\eta)\delta_{\emptyset}(d\zeta)$ and thus satisfies the usual conditions.
 The characterization of convergence in $\Gamma_0$ and continuity of $D_t$ imply that for each $F \in C(\Gamma_0)$ also $L(t)F(\eta)$ is continuous in $(t,\eta)$.
 Concerning (A), fix $\e > 0$, $T > 0$ and a compact
 $B \subset \Gamma_0$. Then there exist $\delta_B > 0$, $N_B \in \N$ and a compact $\Lambda_B \subset \R^d$ such that for each $\eta \in B$
 \begin{align}\label{COMPACTB}
  |\eta| \leq N_B, \quad \eta \subset \Lambda_B, \quad \forall x,y \in \eta,\ x \neq y:\ |x-y| \geq \delta_B
 \end{align}
 holds. Let $A \subset \Gamma_0$ be a compact of the form \eqref{COMPACT} with $\delta, N, \Lambda$ as in \eqref{COMPACTB}.
 Then for each $\eta \in B$ and $\xi \subset \eta$ we obtain that \eqref{COMPACTB} also holds
 for $\eta \backslash \xi$ instead of $\eta$. Hence $\eta \backslash \xi \in A$ and thus $Q(t,\eta,A^c) = 0$ for any $t \in [0,T]$.
\end{proof}

\subsubsection*{The BDLP-model}
In \cite{BP97, BP99, DL00, DL05} the so called Bolker-Dieckmann-Law-Pacala model (short BDLP-model) was introduced to study spatial patterns for certain ecological systems.
Elements $x \in \eta$ are interpreted as plants and the configuration $\eta \in \Gamma_0$ describes therefore the whole ecological system.
The BDLP-model is based only on the two elementary events $\eta \longmapsto \eta \cup x$ (branching of plants) and $\eta \longmapsto \eta \backslash x$ (death of plants).
The branching is assumed to be density independent, that is any plant at position $x \in \eta$ creates with intensity $0 \leq \lambda \in C(\R_+ \times \R^d)$
a new plant at position $y \in \R^d \backslash \eta$ and the spatial probability distribution for the new plant is given by $a^+(x,y)dy$, where
$a^+ \in C(\R^d \times \R^d)$. Moreover, each plant at position $x \in \eta$ has an individual lifetime independent of the other plants. Such lifetime is described
by the intensity $0 \leq m \in C(\R_+ \times \R^d)$. The competition between different plants is assumed to be of additive type and hence of the form
$\sum _{y \in \eta \backslash x}a^-(x,y)$, where $0 \leq a^- \in C(\R^d \times \R^d)$ is the competition kernel. Above description is summarized
in the form of the following Markov (pre-)generator
\begin{align*}
 (L(t)F)(\eta) &= \sum _{x \in \eta}\Big( m(t,x) + \sum _{y \in \eta \backslash x}a^-(x,y)\Big)
 (F(\eta \backslash x) - F(\eta))
 \\ & \ \ \ + \sum _{x \in \eta}\lambda(t,x)\int _{\R^d}a^+(x,y)(F(\eta \cup y) - F(\eta))\,dy.
\end{align*}
Such model has been analyzed in the time-homogeneous case in \cite{FM04}.
In applications one is often interested in $a^+$ being of the form
\[
 a^+(x,y) \sim \frac{1}{|x-y|^{\alpha}}, \quad |x-y| \to \infty
\]
or
\[
 a^+(x,y) \sim e^{-\nu |x-y|^{\alpha}}, \quad |x-y| \to \infty.
\]
\begin{Theorem}
 Suppose that $m, \lambda, a^-$ are continuous and bounded, $a^+$ is continuous with $1 = \int _{\R^d}a^+(x,y)dy$
 and for any compact $\Lambda \subset \R^d$ there exists $a^* \geq 0$ with $a^* \in L^1(\R^d)$ such that
 \[
  a^+(x,y) \leq a^*(y), \quad x \in \Lambda, \quad y \in \R^d
 \]
 holds. Then conditions (A)--(D) hold for $V(\eta) = |\eta| + |\eta|^2$.
\end{Theorem}
\begin{proof}
 Let $B \subset \Gamma_0$ be a compact and take $N_B \in \N$, $\Lambda_B \subset \R^d$ and $\delta_B > 0$ like in \eqref{COMPACT}. Let $A \subset \Gamma_0$ be another
 compact defined by \eqref{COMPACT} with $N_A := N_B + 1$, $\Lambda_B \subset \Lambda_A$ and $\delta_A \in (0, \delta_B)$, then $B \subset A$ holds. We obtain for
 $x \in \Lambda_B$ and $\eta \in B$
 \[
  \int _{\R^d}\one_{A^c}(\eta \cup y)a^+(x,y)\,dy \leq \int _{\Lambda_A^c}a^+(x,y)\,dy
  + \int _{B_{\delta_A}(\eta)}a^+(x,y)\,dy,
 \]
 where $B_{\delta_A}(\eta) := \{ w \in \R^d \ | \ \exists y \in \eta: \ |w-y| < \delta_A\}$. Since $\eta \in B$ and $\delta_B > \delta_A$ we obtain
 $B_{\delta_A}(\eta) = \bigsqcup _{y \in \eta}B_{\delta_A}(y) \subset \Lambda_B^{\delta_B}$ where
 $\Lambda_B^{\delta_B} := \{ w \in \R^d \ | \ d(w, \Lambda_B) \leq \delta_B\}$ with $d(w,\Lambda_B) := \inf\ \{ |w-u| \ | \ u \in \Lambda_B\}$.
 Let $c > 0$ be such that $a^+(x,y) \leq c$ for all $x \in \Lambda_B$ and $y \in \Lambda_B^{\delta_B}$, then
 \[
  \int _{\R^d}\one_{A^c}(\eta \cup y)a^+(x,y)\,dy \leq \int _{\Lambda_A^c}a^*(y)\,dy + N_B c |B_{\delta_A}|
 \]
 is satisfied, where $|B_{\delta_A}|$ is the Lebesgue volume of $B_{\delta_A} = \{ w \in \R^d \ | \ |w| \leq \delta_A\}$. Condition (A) now follows from above estimate,
 Lemma \ref{DEATHLEMMA} and $\lambda \in C_b(\R_+ \times \R^d)$. Condition (B) follows from
 \begin{align*}
  (L(t)V)(\eta) &= \sum _{x \in \eta}\lambda(t,x) + 2 \sum _{x \in \eta}\sum _{y \in \eta \backslash x}a^-(x,y)
   \\ &\ \ \ + 2|\eta|\sum _{x \in \eta}(\lambda(t,x) - m(t,x)) - 2|\eta|\sum _{x \in \eta}\sum _{y \in \eta \backslash x}a^-(x,y)
   \\ &\leq \max\{ \Vert \lambda\Vert_{\infty}, 2 \Vert a^- \Vert_{\infty} + 2 \Vert \lambda\Vert_{\infty} + 2 \Vert m \Vert_{\infty}\} V(\eta).
 \end{align*}
 Condition (D) is fulfilled due to
 \begin{align*}
  q(t,\eta) &= \sum _{x \in \eta}m(t,x) + \sum _{x \in \eta}\lambda(t,x) + \sum _{x \in \eta}\sum _{y \in \eta \backslash x}a^-(x,y)
  \\ &\leq \max\{ \Vert m \Vert_{\infty} + \Vert \lambda \Vert_{\infty}, \Vert a^- \Vert_{\infty}\} V(\eta).
 \end{align*}
 For condition (C) it is enough to show that for any continuous function $F$ such that
 $|F(\eta)| \leq \Vert F \Vert_V (1 + |\eta| + |\eta|^2)$ also $(t,\eta) \longmapsto \sum _{x \in \eta}\lambda(t,x)\int _{\R^d}a^+(x,y)F(\eta \cup y)dy$
 is continuous. Since $\lambda(t,x)$ is continuous it is enough to show that the integral is continuous. But this follows from dominated convergence and
 the condition imposed on $a^+$.
\end{proof}
Above statement implies the following a priori estimate for the evolution of states.
Let $\mu$ be a probability measure with $\int _{\Gamma_0}(1 + |\eta| + |\eta|^2)\mu(d\eta) < \infty$. Then
\[
 \int _{\Gamma_0}(1+|\eta| + |\eta|^2)U^*(t,s)\mu(d\eta)
 \leq e^{(t-s)c}\int _{\Gamma_0}(1 + |\eta| + |\eta|^2)\mu(d\eta)
\]
holds, $c:= \max\{ \Vert \lambda \Vert_{\infty}, 2 \Vert a^- \Vert_{\infty} + 2 \Vert \lambda \Vert_{\infty} + 2 \Vert m \Vert_{\infty}\}$.

\subsubsection*{Dieckmann-Law model}
In contrast to the BDLP-model we discuss here one possible extension for which the branching mechanism includes interactions of the plants.
For simplicity we suppose that all intensities are translation invariant. A plant at location $x \in \eta$ shall now have the modified branching intensity given by
\[
 \lambda(t) + \sum _{y \in \eta \backslash x}b^+(x-y), \quad t \geq 0,
\]
where $0 \leq b^+ \in C_b(\R^d)$. The location of the offspring is described by the probability density $a^+(x-y)$. The modified Markov (pre-)generator is therefore given by
\begin{align*}
 (L(t)F)(\eta) &= \sum _{x \in \eta}\Big(m(t) + \sum _{y \in \eta \backslash x}a^-(x-y)\Big)(F(\eta \backslash x) - F(\eta))
 \\ & \ \ + \sum _{x \in \eta} \lambda(t)\int _{\R^d}(F(\eta \cup w) - F(\eta))a^+(x-y)\,dw
 \\ & \ \ + \sum _{x \in \eta}\sum _{y \in \eta \backslash x}b^+(x-y)\int _{\R^d}(F(\eta \cup w)
 - F(\eta))a^+(x-w)\,dw,
\end{align*}
where $m,\lambda \in C(\R_+)$ and $a^- \in C_b(\R^d)$.
We assume that $a^- - b^+$ is a stable potential. By definition this means that there exists a constant $b \geq 0$ such that
\[
 \sum _{x \in \eta}\sum _{y \in \eta \backslash x}(a^-(x-y) - b^+(x-y)) \geq - b|\eta|, \quad \eta \in \Gamma_0.
\]
Let $E^+(\eta) = \sum _{x \in \eta } \sum _{y \in \eta \backslash x}b^+(x-y)$ and
$E^-(\eta) = \sum _{x \in \eta}\sum _{y \in \eta \backslash x}a^-(x-y)$, that it above condition is equivalent to
\[
 E^+(\eta) \leq b|\eta| + E^-(\eta), \quad \eta \in \Gamma_0.
\]
\begin{Theorem}\label{IPSFINITETH:03}
 Suppose that for any compact $\Lambda \subset \R^d$ there exists $a^* \in L^1(\R^d)$ which satisfies
 \[
  a^+(x-w) \leq a^*(w), \quad x \in \Lambda,\quad w \in \R^d.
 \]
 Then conditions (A)--(D) are satisfied for $V(\eta) := |\eta| + |\eta|^2$.
 Moreover, for any $n \geq 1$ and state $\mu$ with $\int _{\Gamma_0}|\eta|^n \mu(d\eta) < \infty$, the evolution of states satisfies
 $\int _{\Gamma_0}|\eta|^n U^*(t,s)\mu(d\eta) < \infty$.
 If in addition $m(t), \lambda(t) > 0$ for all $t \geq 0$, then condition (E) holds and $U^*(t,s)$ leaves the space of densities invariant.
\end{Theorem}
\begin{proof}
 Condition (A) will be shown for a more general case later on. Concerning condition (B) we have
 \[
  (L(t)|\cdot|)(\eta) \leq (b + \lambda(t) - m(t))|\eta|
 \]
 and by $(|\eta| + 1)^2 - |\eta|^2 = 2|\eta| - 1$, $(|\eta| - 1)^2 - |\eta|^2 = -2 |\eta|$ also
 \[
  (L(t)|\cdot|^2)(\eta) \leq (2\lambda(t) + \Vert b^+ \Vert_{\infty} + 2 b - 2m(t))|\eta|^2 + (\lambda(t) - \Vert b^+\Vert_{\infty})|\eta|.
 \]
 Altogether this yields
 \[
  L(t)V(\eta) \leq |\eta|(b + 2 \lambda(t) - m(t) - \Vert b^+ \Vert_{\infty}) + |\eta|^2 ( 2 \lambda(t) + \Vert b^+ \Vert_{\infty} + 2 b - 2m(t)),
 \]
 i.e. \eqref{PHD:08} is satisfied. Since
 \begin{align*}
  q(t,\eta) &= (m(t) + \lambda(t))|\eta| + E^+(\eta) + E^-(\eta)
  \\ &\leq (\Vert a^- \Vert_{\infty} + \Vert b^+ \Vert_{\infty})|\eta|^2 + |\eta| \sup _{t \in [0,T]}(m(t) + \lambda(t))
 \end{align*}
 also (D) holds. For property (C) it is enough to show that $x \longmapsto \int _{\R^d}F(\eta \cup y) a^+(x-y)\,dy$ is continuous for any continuous function $F$
 with $|F(\eta)| \leq c(1 + |\eta| + |\eta|^2)$, $\eta \in \Gamma_0$ and some constant $c > 0$. But this follows immediately by dominated convergence and the
 assumptions on $a^+$. Property (E) is a direct consequence of the continuity of $m$ and $\lambda$. For the remaining assertion it suffices to show that for any $n \geq 1$
 there exist a continuous function $c_n: \R_+ \longmapsto \R_+$ such that
 \[
  (L(t)|\cdot|^n)(\eta) \leq c_n(t)|\eta|^n, \quad t \geq 0.
 \]
 We have $(|\eta|+1)^n - |\eta|^n = \sum _{l=0}^{n-1}\binom{n}{l}|\eta|^l$, $(|\eta|-1)^n - |\eta|^n = \sum _{l=0}^{n-1}\binom{n}{l}(-1)^{n-l}|\eta|^k \leq 0$
 and since $(L(t)|\cdot|^n)(\emptyset) = 0$ we can assume w.l.g. that $|\eta| > 0$. Hence
 \begin{align*}
  (L(t)|\cdot|^n)(\eta) &\leq  \lambda(t)\sum _{l=0}^{n-1}\binom{n}{l}|\eta|^{l+1} + \sum _{l = 0}^{n-1}\binom{n}{l}|\eta|^l(E^+(\eta) + (-1)^{n-l}E^-(\eta))
  \\ &= \lambda(t)\sum _{l=1}^{n}\binom{n}{l-1}|\eta|^l + \sum _{l=1}^{n}\binom{n}{l-1}|\eta|^{l-1}(E^+(\eta) - (-1)^{n-l}E^-(\eta))
  \\ &\leq |\eta|^n \lambda(t)\sum _{l=1}^{n}\binom{n}{l-1} + \sum _{l=1}^{n-1}\binom{n}{l-1}|\eta|^{n}(\Vert b^+ \Vert_{\infty} + \Vert a^- \Vert_{\infty})
  \\ & \ \ \ \ + \binom{n}{n-1}(E^+(\eta) - E^-(\eta))|\eta|^{n-1}
  \\ &\leq 2^n (\lambda(t) + \Vert b^+ \Vert_{\infty} + \Vert a^- \Vert_{\infty})n\cdot |\eta|^n + b n |\eta|^n
 \end{align*}
 implies the assertion.
\end{proof}
\begin{Remark}
 The proof shows that $U^*(t,s)$ maps the space of probability measures with the constraint $\int _{\Gamma_0}|\eta|^n \mu(d\eta) < \infty$ continuously
 on itself. Moreover, using Corollary \ref{GJPTH:05} one can show that
 \[
  \int _{\Gamma_0}|\eta|U^*(t,s)\mu(d\eta) \leq e^{b(t-s)} e^{\int _{s}^{t}(\lambda(r) - m(r))dr}\int _{\Gamma_0}|\eta|\mu(d\eta)
 \]
 and
 \begin{align*}
  \int _{\Gamma_0}(|\eta| + |\eta|^2) U^*(t,s)\mu(d\eta) &\leq e^{(b - \Vert b^+ \Vert_{\infty})(t-s)} e^{ \int _{s}^{t}(2 \lambda(r) - m(r))dr} \int _{\Gamma_0}|\eta|\mu(d\eta)
  \\ & \ \ \ + e^{(\Vert b^+ \Vert_{\infty} + 2b)(t-s)}e^{2\int _{s}^{t}(\lambda(r) - m(r))dr}\int _{\Gamma_0}|\eta|^2 \mu(d\eta)
 \end{align*}
 are valid.
\end{Remark}

\subsubsection*{Generalized Dieckmann-Law model}
Assume that any plant at position $x \in \eta$ may create any number $k \in \N$ of new plants.
Their locations are, for any fixed $t \geq 0$, distributed according to the probability measure
\[
 a^+(t,x - y_1) \cdots a^+(t,x - y_k)\,dy_1 \cdots dy_k.
\]
Therefore the (pre-)generator is assumed to be given by
\begin{align*}
 (L(t)F)(\eta) &= \sum _{x \in \eta}\Big(m(t,x) + \sum _{y \in \eta \backslash x}a^-(t,x-y)\Big)(F(\eta \backslash x) - F(\eta))
 \\ & \ \ + \frac{1}{e}\sum _{x \in \eta} \lambda(t,x)\int _{\Gamma_0 \backslash
 \{\emptyset\}}(F(\eta \cup \zeta) - F(\eta))e_{\lambda}(a^+(t,x-\cdot);\zeta)\,d\lambda(\zeta)
 \\ & \ \ + \frac{1}{e}\sum _{x \in \eta}\sum _{y \in \eta \backslash x}b^+(t,x-y)
 \int _{\Gamma_0 \backslash \{\emptyset\}}(F(\eta \cup \zeta) - F(\eta))e_{\lambda}(a^+(t,x-\cdot);\zeta)\,d\lambda(\zeta).
\end{align*}
The factor $\frac{1}{e}$ is a normalization factor since we have
\[
 \int _{\Gamma_0}e_{\lambda}(a^+(t,x-\cdot);\zeta)\,d\lambda(\zeta) = e.
\]
\begin{Theorem}\label{IPSFINITETH:04}
 Let $0 \leq m, \lambda, a^{-}, b^+ \in C_b(\R_+ \times \R^d)$ with $a^+(t,\cdot)$ being a probability density for all $t \geq 0$.
 Suppose that for any compact $\Lambda \subset \R^d$ and $T > 0$ there exists $a^* \in L^1(\R^d)$ which satisfies
 \begin{align}\label{IPSFINITE:03}
  a^+(t,x-y) \leq a^*(y), \quad x \in \Lambda, \quad t \in [0,T], \quad y \in \R^d.
 \end{align}
 Moreover, assume that $b^+(t,x) \leq a^-(t,x)$ holds for all $x \in \R^d$ and $t \geq 0$.
 Then conditions (A)--(D) are satisfied for $V(\eta) = |\eta| + |\eta|^2$.
\end{Theorem}
\begin{proof}
By $\int _{\Gamma_0 \backslash \emptyset}|\zeta|e_{\lambda}(a^+(t,x-\cdot);\zeta)d\lambda(\zeta) = e$ we obtain
\begin{align*}
 L(t)V(\eta) &= \sum _{x \in \eta}\left( (2 - e^{-1})\lambda(t,x) - 2m(t,x)\right)
 \\ &\ \ \ + \sum _{x \in \eta}\sum _{y \in \eta \backslash x}\left( (2 - e^{-1})b^+(t,x-y) - 2a^-(t,x-y)\right)
 \\ &\ \ \ + 2|\eta|\sum _{x \in \eta}(\lambda(t,x) - m(t,x)) + 2 |\eta| \sum _{x \in \eta}\sum _{y \in \eta \backslash x}(b^+(t,x-y) - a^-(t,x-y))
 \\ &\leq 2 (\Vert \lambda \Vert_{\infty} + \Vert m \Vert_{\infty})V(\eta),
\end{align*}
which implies condition (B). Condition (D) follows from
\begin{align*}
 q(t,\eta) &= \sum _{x \in \eta}m(t,x) + \frac{e-1}{e}\sum _{x \in \eta}\lambda(t,x)
 \\ &\ \ \ + \sum _{x \in \eta}\sum _{y \in \eta \backslash x}a^-(t,x-y) + \frac{e-1}{e}\sum _{x \in \eta}\sum _{y \in \eta \backslash x}b^+(t,x-y)
 \\ &\leq (\Vert m \Vert_{\infty} + \Vert \lambda \Vert_{\infty})|\eta| + (\Vert a^- \Vert_{\infty} + \Vert b^+ \Vert_{\infty}) |\eta|^2.
\end{align*}
In order to see (C), observe that the assertion is clear for the contribution from the terms of the operator $L(t)$ describing the death of plants.
Because $\lambda$ and $b^+$ are continuous it suffices to show for any $F \in C_V(\Gamma_0), \eta_n \to \eta, t_n \to t$ and $x_n \in \eta_n, x \in \eta$ with $x_n \to x$
\begin{gather*}
 \int _{\Gamma_0 \backslash \emptyset}F(\eta_n \cup \zeta) e_{\lambda}(a^+(t_n, \cdot - x_n);\zeta)\,d\lambda(\zeta)
 \to \int _{\Gamma_0 \backslash \emptyset}F(\eta \cup \zeta) e_{\lambda}(a^+(t, \cdot - x);\zeta)\,d\lambda(\zeta),\\
 \quad n \to \infty.
\end{gather*}
Since the integrand is continuous it converges for each
$\zeta \in \Gamma_0 \backslash \emptyset$ and by \eqref{IPSFINITE:03} with compacts $K = \{t_n \ | \ n \geq 1\} \cup \{t\}$, $B = \{x_n\ |\ n \geq 1\} \cup \{x\}$ we obtain
by dominated convergence the assertion. Therefore it remains to show property (A). Take $T > 0$ and fix a compact $B \subset \Gamma_0$.
Hence there exists $\Lambda_B \subset \R^d$ compact, $N_B \in \N$ and $\delta_B > 0$ such that for any $\eta \in B$ \eqref{COMPACTB} holds.
Condition (A) was shown for the death of plants, so let us focus on the terms contributing to the birth.
Due to the continuity of $\lambda, b^+$ the sum
\[
 \sum _{x \in \eta}\Big( \lambda(t,x) + \sum _{y \in \eta \backslash x}b^+(t,x-y)\Big)
\]
is uniformly bounded on $[0,T]\times B$. Hence it is enough to estimate the integral. Take a compact set $A \subset \Gamma_0$ with the characteristics $N_A > N_B$,
$\delta_A < \delta_B$, $\Lambda_B \subset \Lambda_A$, i.e. \eqref{COMPACT} and set
\[
 B_{\delta_A}(\eta) = \Big\{ \xi \in \Gamma_0 \ | \ \xi \subset \bigcup _{x \in \eta}B_{\delta_A}(x)\Big\},
\]
where $B_{\delta_A}(x) = \{ y \in \R^d\ | \ |x-y| < \delta_A\}$. Then we obtain for $\eta \in B$ and $x \in \eta$, so $x \in \Lambda_B$
\begin{align*}
 & \int _{\Gamma_0 \backslash \emptyset}\one_{A^c}(\eta \cup \zeta)e_{\lambda}(a^+(t,x - \cdot);\zeta)\,d\lambda(\zeta)
 \\ &\leq \Big(\int _{|\zeta| > N_A-N_B} + \int _{B_{\delta_A}(\eta)\backslash \emptyset} +
 \int _{\Gamma_{\Lambda_A^c} \backslash \emptyset} + \int _{C(\delta_A)}\Big)e_{\lambda}(a^+(t,x - \cdot);
 \zeta)\,d\lambda(\zeta)\\
 &= I_1 + I_2 + I_3 + I_4 ,
\end{align*}
where $C(\delta_A) = \{ \zeta \in \Gamma_0 \ | \ \exists w \neq z, \ w,z \in \zeta :\ |w-z| < \delta_A \}$.
For the first integral we obtain uniformly in $t \in [0,T], \eta \in B$ and $x \in \eta$
\begin{align*}
 I_1 \leq \int _{|\zeta| > N_A-N_B}e_{\lambda}(a^*;\zeta)\,d\lambda(\zeta) =
 \sum _{n = N_A - N_B + 1}^{\infty}\frac{\Big( \int _{\R^d}a^*(y)dy\Big)^n}{n!}
\end{align*}
and similarly for the third
\begin{align*}
 I_3 \leq \int _{\Gamma_{\Lambda_A^c} \backslash \emptyset}e_{\lambda}(a^*;\zeta)\,d\lambda(\zeta) =
 \sum _{n=1}^{\infty}\frac{1}{n!}\Big( \int _{\Lambda_A^c}a^*(y)dy\Big)^n = \exp\Big( \int
 _{\Lambda_A^c}a^*(y)dy\Big) - 1.
\end{align*}
This two terms tend uniformly in $\eta \in B$ and $t \in [0,T]$ to zero as $N_A \to \infty$ and $\Lambda_A \to \R^d$. Denote by $c > 0$ a constant for which
\[
 a^+(t,z-w) \leq c, \quad t \in [0,T], \quad z \in \Lambda_B, \quad w \in \Lambda_B^{\delta_B}
\]
with $\Lambda_B^{\delta_B} = \{ w \in \R^d \ | \ d(w, \Lambda_{B}) \leq \delta_B\}$ holds, where
$d(w, \Lambda_B) := \inf \{ d(w,u)\ | \ u \in \Lambda_B\}$. For $I_2$ we obtain with $|B_{\delta_A}|$ the Lebesgue volume of
a ball with radius $\delta_A$ in $\R^d$, since for any $w,z \in \eta$ with $w \neq z$: $B_{\delta_A}(w) \cap B_{\delta_A}(z) = \emptyset$
\begin{align*}
 I_2 = \int _{B_{\delta_A}(\eta)\backslash \emptyset}e_{\lambda}(a^+(t,x-\cdot);\zeta)\,d\lambda(\zeta)
 \leq \Big(\sum _{n=1}^{\infty}\frac{c^n |B_{\delta_A}|^n}{n!}\Big)^{|\eta|} = (e^{c|B_{\delta_A}|}-1)^{|\eta|}.
\end{align*}
Finally due to $C(\delta_A) \to \emptyset$ as $\delta_A \to 0$ we have shown that for all $T > 0$, all compacts $B \subset E$ and $\e > 0$ there
is a compact $A \subset E$ such that
\[
 Q(t,\eta,A^c) < \e, \quad t \in [0,T],\quad \eta \in B,
\]
which is stronger then (A).
\end{proof}
\begin{Remark}
Condition \eqref{IPSFINITE:03} is for instance satisfied if there exist strictly positive continuous functions $\lambda, C > 0$ and $R > 0$, $\alpha > \frac{d}{2}$ such that
\[
 a^+(t,x) \leq \frac{C(t)}{(\lambda(t) + |x|^2)^{\alpha}}, \quad |x| \geq R
\]
holds.
\end{Remark}
\begin{Remark}
 In the time-homogeneous case weaker conditions are sufficient to prove the Feller property.
\end{Remark}

\section*{Appendix}
Set $\Delta := \{ (s,t) \in \R_+ \times \R_+ \ | \ s \leq t \}$, the next two lemmas should be well-known and are included here only for convenience.
\begin{Lemma}\label{LEMMAFELLER}
 Let $f_j: \Delta \times \R_+ \times E \longrightarrow \R$ be a family of measurable functions indexed by $j \in M$, where $M$ is an arbitrary non-empty index set, such that
 \begin{enumerate}
  \item[]{1.} $f_j$ is bounded on compacts uniformly in $j \in M$.
  \item[]{2.} The map $(s,t,x) \longmapsto f_j(s,t,r,x)$ is continuous uniformly in $j \in M$ for fixed $r \in [s,t]$.
 \end{enumerate}
 Then $(s,t,x) \longmapsto \int _{s}^{t}f_j(s,t,r,x)\,dr$ is continuous uniformly in $j \in M$.
\end{Lemma}
\begin{proof}
 Let $(s,t),(s_n,t_n) \in \Delta$ and $x,x_n \in E$ be such that $s_n \to s$, $t_n \to t$ and $x_n \to x$ as $n \to \infty$. We find $T > 0$
 and a compact $B \subset E$ such that $s,s_n,t,t_n \in [0,T]$ and $x,x_n \in B$ for $n \in \N$.
 Let $f^* := \sup _{j \in M}\sup _{(t_1,t_2,t_3,x) \in \Delta \cap [0,T]^2 \times [0,T] \times B}\ f_j(t_1,t_2,t_3,x) < \infty$, then for any $n \in \N$ and $j \in M$
 \begin{align*}
  & \Big| \int _{s}^{t}f_j(s,t,r,x)\,dr - \int _{s_n}^{t_n}f_j(s_n,t_n,r,x_n)\,dr\Big|
  \\ &\leq |s - s_n| f^* + |t - t_n| f^* + \int _{0}^{T}|f_j(s,t,r,x) - f_j(s_n,t_n,r,x_n)|\,dr.
 \end{align*}
 For each $r \in [0,T]$ the integrand on the right-hand-side tends to zero as $n \to \infty$, and since $|f_j(s,t,r,x) - f_j(s_n,t_n,r,x_n)| \leq 2 f^*$
 dominated convergence yields the assertion.
\end{proof}

The next lemma will show continuity in the case where instead of $dr$ there is an arbitrary kernel $H(t,x,dy)$. In such a case we will need that $E$ is locally compact.

\begin{Lemma}\label{LEMMAFELLER2}
 Let $E$ be a locally compact Polish space,
 \[
  f: \{ (s,r,t) \in \R_+^3\ | \ s \leq r \leq t\} \times E \times E \longrightarrow \R
 \]
 be continuous and bounded, and let
 $H: I \times E \times \mathcal{B}(E) \longrightarrow \R_+$ be a weakly continuous kernel, i.e. for all $F \in C_b(E)$,
 $\R_+ \times E \ni (r,x) \longmapsto \int _{E}F(y)H(r,x, dy)$ is continuous. Then
 \[
  (s,r,t,x) \longmapsto \int _{E}f(s,r,t,x,y)H(r,x,dy)
 \]
 is continuous.
\end{Lemma}
\begin{proof}
 Let $s_n \leq r_n \leq t_n$ be such that $s_n \to s, r_n \to r, t_n \to t$ and $x_n \to x$ as $n \to \infty$. Fix $\e > 0$ and take $A \subset E$ compact with
 $H(r,x,A^c) < \e$. Since $E$ is a locally compact space we can find another compact $A_1 \subset E$ with $A \subset \overset{\circ}{A_1} \subset A_1$.
 Portmanteau implies then $\underset{n \to \infty}{\limsup}\ H(r_n, x_n, (\overset{\circ}{A_1})^c) \leq H(r, x,(\overset{\circ}{A_1})^c) \leq H(r,x, A^c) < \e$.
 The function $f$ restricted to the compact $\{(s_n,r_n,t_n)\ | \ n \in \N\} \cup \{(s,r,t)\} \times \{ x_n \ | \ n \in \N\} \cup \{x\} \times A_1$ is
 uniformly continuous and hence we obtain for sufficiently large $n$
 \begin{align*}
  & \Big| \int _{E}f(s_n,r_n,t_n,x_n,y)H(r_n, x_n,dy) - \int _{E}f(s,r,t,x,y)H(r,x,dy)\Big|
  \\ &\leq \int _{E}|f(s_n,r_n,t_n,x_n,y) - f(s,r,t,x,y)| H(r_n,x_n,dy)
  \\ & \ + \Big| \int _{E}f(s,r,t,x,y)H(r_n,x_n,dy) - \int _{E}f(s,r,t,x,y)H(r, x,dy)\Big|
  \\ &\leq H(r_n,x_n, A_1)\e + 2 \Vert f \Vert H(r_n,x_n,(\overset{\circ}{A_1})^c) + \e
  \\ &\leq H(r_n, x_n)\e + 2\Vert f \Vert \e + \e.
 \end{align*}
 Due to the weak continuity of $H$ the function $H(r,x) := H(r,x,E)$ is continuous and hence $H(r_n,x_n)$ is uniformly bounded in $n\in \N$, which shows the assertion.
\end{proof}

{\it{Acknowledgments.}}  The author would like to thank Viktor
Bezborodov for many fruitful discussions and critical remarks.

\end{document}